\documentclass{article}

\usepackage[margin=1in]{geometry}

\usepackage{amssymb}
\usepackage{amsmath}
\usepackage{amsthm}
\usepackage{enumitem}
\setlist{nolistsep}
\usepackage{bm}
\usepackage{tikz}
\usetikzlibrary{arrows.meta}
\usetikzlibrary{decorations} 
\usetikzlibrary{shapes}
\usepackage[margin=2cm]{caption}
\usepackage{subcaption}
\usepackage{titlesec}
\usepackage[backref]{hyperref}
\usepackage{graphicx}
\usepackage{xcolor}
\usepackage{thmtools}
\usepackage{thm-restate}
\usepackage{cleveref}
\usepackage{lmodern}
\usetikzlibrary{arrows.meta,backgrounds,calc}

\definecolor{setblue}{HTML}{2B6F9C}
\definecolor{setbluefill}{HTML}{EAF3F8}
\definecolor{setorange}{HTML}{B56A22}
\definecolor{setorangefill}{HTML}{FFF2E3}
\definecolor{pathblue}{HTML}{2458A6}
\definecolor{greedy}{HTML}{34424B}
\definecolor{softgray}{HTML}{66737C}
\definecolor{bridgepurple}{HTML}{76569A}
 \setlist{  
  listparindent=\parindent,
  parsep=0pt,
}

\newtheorem{theorem}{Theorem}[section]

\newtheorem{lemma}[theorem]{Lemma}
\newtheorem*{claim*}{Claim}
\newtheorem{claim}{Claim}

\newtheorem{cor}[theorem]{Corollary}
\newtheorem{conjecture}[theorem]{Conjecture}
\newtheorem{prop}[theorem]{Proposition}

\theoremstyle{definition}
\newtheorem{definition}[theorem]{Definition}
\newtheorem{question}[theorem]{Question}

\newenvironment{subproof}[1][Proof]{\begin{proof}[#1]}{\end{proof}}
\newcommand{\spn}{\mathrm{span}}
\newcommand{\eps}{\varepsilon}

\newcommand{\Vect}{\mathrm{Vec}}
\newcommand{\ex}{\mathrm{ex}}

\newcounter{propcounter}

\title{Generic solutions to symmetric linear equations}
\author{Bryce Frederickson\thanks{Department of Mathematics, Emory University, 
Atlanta, GA, USA. Email: {\tt bfrede4@emory.edu}}, Liana Yepremyan\thanks{Department of Mathematics, Emory University, 
Atlanta, GA, USA. Email: {\tt lyeprem@emory.edu}. Research for this project was partially supported by the National Science Foundation grant 2247013: Forbidden and Colored Subgraphs.} }

\date{}

\begin{document}

\maketitle

\begin{abstract}
    In 1993, Ruzsa showed that for every $k \geq 2$, there exists a constant $C$ such that every subset $A \subseteq [N]$ of size at least $C N^{1/k}$ contains $2k$ distinct elements $a_1, \ldots, a_k, b_1, \ldots, b_k \in A$ such that $a_1 + \cdots + a_k = b_1 + \cdots + b_k$. We strengthen this result by proving that the elements $a_1, \ldots, a_k, b_1, \ldots, b_k$ can be chosen to have the additional property that $\{a_1, \ldots, a_k, b_1, \ldots, b_k\}$ has $2^{2k}-1$ distinct subset sums, with the only coincidence being that $\{a_1, \ldots, a_k\}$ and $\{b_1, \ldots, b_k\}$ have the same sum.
    
    Our proof also applies to any finite Abelian group of odd order $N$, and it provides a corresponding supersaturation result: that whenever $|A| \geq CN^{1/k}$, there are at least $\Omega(|A|^{2k}/N)$ choices for $a_1, \ldots, a_k, b_1, \ldots, b_k \in A$ satisfying these properties. We prove a slightly weaker statement for Abelian groups of even order.

    We also apply our methods to the vector space setting and prove the following $\mathbb F_q$-analogue of the Bondy--Simonovits Theorem on the extremal number of even cycles in graphs: Any rank-$n$, simple, $\mathbb F_q$-representable matroid with no circuit of size exactly $2k$ has size at most $C q^{n/k}$ for some constant $C$ depending only on $q$ and $k$. When $q=2$, this is best possible up to the constant $C$ for all $k \geq 2$. Our methods also apply to the original graph setting and give a new proof of the Bondy--Simonovits Theorem and its supersaturation version.

\end{abstract}

\section{Introduction}
\subsection{Symmetric equations}\label{sec:symmetric equations}

This article primarily concerns \emph{symmetric} linear equations in $2k$ variables $x_1, \ldots, x_{2k}$, which take the form
\[\label{eq:symmetric equation}
\lambda_1 x_1 + \cdots + \lambda_k x_k = \lambda_1 x_{k+1} + \cdots + \lambda_k x_{2k}, \tag{\textbf{SLE}}\]
where the coefficients $\lambda_1, \ldots, \lambda_k$ may come from the integers or from a (finite) field. Our central objective is to explore the minimum required density of a subset $A$ in $[N]$, in a finite Abelian group $\Gamma$, or in a finite vector space $\mathbb F_q^n$, in order to guarantee ``nondegenerate'' solutions to~\eqref{eq:symmetric equation} with $x_i \in A$ for all $i \in [2k]$, where we consider various notions of ``nondegeneracy''.

The most fundamental symmetric linear equation is the \emph{symmetric additive equation}, in which $\lambda_1 = \cdots = \lambda_k = 1$, namely
\[\label{eq:equal k-sums} x_1 + \cdots + x_k = x_{k+1} + \cdots + x_{2k}.\tag{\textbf{SAE}}\]
Clearly, any nonempty subset $A$ of any Abelian group $\Gamma$ has solutions $(x_1, \ldots, x_{2k}) \in A^{2k}$ to~\eqref{eq:equal k-sums} by taking $(x_{k+1}, \ldots, x_{2k})$ to be any permutation of $(x_1, \ldots, x_k)$. Such solutions are called \emph{trivial}, and a set $A$ which contains only trivial solutions to~\eqref{eq:equal k-sums} is called a \emph{$B_k$}-set. The case $k=2$ is the most widely studied, and $B_2$-sets are usually called \emph{Sidon sets}. For a subset $A \subseteq \Gamma$, we define the \emph{$k$-fold sumset} of $A$ to be the set
\[A^{+k} := \{a_1 + \cdots + a_k : a_1, \ldots, a_k \in A\} \subseteq \Gamma.\]
In a~$B_k$-set~$A$, any $k$-element multiset in $A$ is determined by its sum, so we necessarily have $\binom{|A|+k-1}{k} \leq |A^{+k}|$.
In particular, if $A$ is a subset of $[N]$ or of a finite Abelian group of order $N$, then $|A| = O(N^{1/k})$.\footnote{Here and throughout, asymptotics are with respect to $N$ if we are considering subsets of $[N] \subseteq \mathbb Z$ or subsets of a finite Abelian group (or vector space) of order $N$. The parameter $k$ is always considered to be a constant.} A foundational construction of Bose and Chowla~\cite{bose-chowla1960} from~$1960$ shows that this simple bound is optimal in the integer setting, up to the leading constant.
\begin{theorem}[Bose, Chowla~\cite{bose-chowla1960}]\label{thm:bose chowla}
    For every prime power $q$ and integer $k \geq 2$, there exists a $B_k$-set in~$\mathbb Z_{q^{k}-1}$ of size $q$. Consequently, for any positive integer $N$, there exists a $B_k$-set in $[N]$ of size $(1-o(1))N^{1/k}$.
\end{theorem}

We observe a similar lower bound for vector spaces over finite fields with characteristic greater than $k$. The construction is basic, but we have not been able to find prior mention of it except in the Sidon set case $k=2$ (see Example~$1$ in~\cite{cilleruelo2012combinatorial}\footnote{In~\cite{cilleruelo2012combinatorial}, the assumption that $\mathbb F_q$ has odd characteristic is not stated, but it is necessary.}).
\begin{prop}\label{thm:Bk construction vector space}
Let $k \geq 2$ be a fixed integer. For every prime $p > k$ and integer $r \geq 0$, there exists a $B_k$-set in $\mathbb F_p^{rk}$ of size $p^r$. Consequently, over every finite field $\mathbb F_q$ of characteristic greater than $k$, for every positive integer $n$, there exists a $B_k$-set in $\mathbb F_q^n$ of size $\Omega(N^{1/k})$, where $N := q^n$.
\end{prop}

A slight weakening of the defining property of a $B_k$-set is to allow for nontrivial solutions to~\eqref{eq:equal k-sums}, provided that at least two variables $x_i$ and $x_j$ are assigned the same value, as follows.

\begin{definition}
    Let $\Gamma$ be an Abelian group. A subset $A \subseteq \Gamma$ is called a \emph{$B_k^*$-set} if there are no solutions $(x_1, \ldots, x_{2k}) \in A^{2k}$ to the equation~\eqref{eq:equal k-sums}
    in which $x_1, \ldots, x_{2k}$ are distinct elements of $A$.
\end{definition}

In 1993, Ruzsa~\cite{ruzsa1993solving} proved that the same asymptotic bound of $O(N^{1/k})$ which holds for $B_k$-sets in $[N]$ also holds for $B_k^*$-sets.
\begin{theorem}[Ruzsa~\cite{ruzsa1993solving}]\label{thm:ruzsa}
If $A \subseteq [N]$ is a $B_k^*$-set, then $|A| \leq (1 + o(1))k^{2-1/k}N^{1/k}$.
\end{theorem}
As we will discuss further in the concluding remarks, there is a well-known connection between $B_k^*$-sets in Abelian groups and the so-called rainbow Tur\'an problem for cycles of length $2k$, and a result of Janzer~\cite{janzer2023rainbow} in this latter setting implies that the conclusion of \Cref{thm:ruzsa} holds for any finite Abelian group. That is, the maximum size of a $B_k^*$-set in an Abelian group of order $N$ is $O(N^{1/k})$.

We will consider further weakenings of $B_k^*$-sets which forbid only those solutions to~\eqref{eq:equal k-sums} which demonstrate even stricter notions of nondegeneracy. The first notion we call \emph{inseparability}, which means that the solution to a system of equations, in addition to assigning variables to distinct elements, also fails to satisfy all proper truncations of the original system.

\begin{definition}
    Let $\Gamma$ be an Abelian group or a vector space (or more generally, an $R$-module for some commutative ring $R$). We say that a solution $(x_1, \ldots, x_k) \in \Gamma^k$ to a system
    \[\left \{ \begin{array}{c c c}\lambda_{1,1} x_1 + \cdots + \lambda_{1,k} x_k & = & 0 \\
    & \vdots & \\
    \lambda_{m,1} x_1 + \cdots + \lambda_{m,k} x_k & = & 0
\end{array}\right.\]
of linear equations is \emph{inseparable} if the following hold:
    \stepcounter{propcounter}
    \begin{enumerate}[label = {{{\normalfont\textbf{\Alph{propcounter}\arabic{enumi}}}}}]
        \item \label{prop:distinct generic inseparable} The elements $x_1, \ldots, x_{k} \in \Gamma$ are distinct.
        \item\label{prop:different subset sums generic inseparable} For every proper nonempty subset $S \subseteq [k]$, we have $\sum_{i \in S} \lambda_{t,i} x_i \neq 0$ for some $t \in [m]$.
    \end{enumerate}
\end{definition}

\noindent
\textbf{Remark.}
The reason for the name \emph{inseparable} comes from the fact that~\ref{prop:different subset sums generic inseparable} is equivalent to the condition that the set $[k]$ cannot be partitioned into nonempty subsets $S,S'$ satisfying
        \[\left \{ \begin{array}{c c c}\sum_{i \in S}\lambda_{1,i} x_i  & = & 0 \\
    & \vdots & \\
    \sum_{i \in S} \lambda_{m,i} x_i & = & 0
\end{array}\right. \qquad \text{and} \qquad
\left \{ \begin{array}{c c c}\sum_{i \in S'}\lambda_{1,i} x_i  & = & 0 \\
    & \vdots & \\
    \sum_{i \in S'} \lambda_{m,i} x_i & = & 0.
\end{array}\right.\]
We also note that in the special case that the system consists of the single equation~\eqref{eq:equal k-sums}, a solution with $x_1, \ldots, x_{2k}$ distinct is inseparable if, for every pair of subsets $S,T \subseteq [k]$, we have $\sum_{i \in S} x_i \neq \sum_{j \in T} x_{k+j}$ unless $S = T = \emptyset$ or $S = T = [k]$.

The second (stricter) notion of nondegeneracy we call \emph{additive genericity}, and we only define it for equations of a certain form. Intuitively, this notion requires that the set of elements involved in the solution be as close to being dissociated as possible, in the sense that the only subset sums that coincide are those required by the equation itself.

\begin{definition}
    Let $\Gamma$ be an Abelian group, let $k \geq 1$ be an integer, and let $\bar \sigma = (\sigma_1, \ldots, \sigma_k) \in \{\pm 1\}^k$ be a sequence of signs. We say that a solution $(x_1, \ldots, x_{k}) \in \Gamma^{k}$ to the equation
    \[\sigma_1 x_1 + \cdots + \sigma_k x_k = 0\]
    is \emph{additively generic} if, for every $\bar \tau = (\tau_1, \ldots, \tau_k) \in \{-1,0,1\}^k \setminus \{0, \pm \bar \sigma\}$, we have $\sum_{i=1}^k \tau_i x_i \neq 0$.
\end{definition}
\noindent
\textbf{Remark.} If we exclude the case $k=2$ and $\bar \sigma = \pm (1, -1)$, this condition is equivalent to saying that the set $\{x_1, \ldots, x_k\}$ has $2^k-1$ distinct subset sums, and in particular, it implies that $x_1, \ldots, x_{k}$ must be distinct.

\begin{definition}
    Let $\Gamma$ be an Abelian group. A subset $A \subseteq \Gamma$ is called a \emph{$B_k^{**}$-set} if there are no inseparable solutions $(x_1, \ldots, x_{2k}) \in A^{2k}$ to the equation~\eqref{eq:equal k-sums}. Similarly, $A$ is called a \emph{$B_k^{***}$-set} if there are no additively generic solutions $(x_1, \ldots, x_{2k}) \in A^{2k}$ to~\eqref{eq:equal k-sums}.
\end{definition}
Note that for $k \geq 2$, any $B_k^*$-set is also a $B_k^{**}$-set, and any $B_k^{**}$-set is also a $B_k^{***}$-set. Our first main result states that the bound of $O(N^{1/k})$ given by Ruzsa (\Cref{thm:ruzsa}) for the maximum size of a $B_k^*$-set in~$[N]$ also holds for $B_k^{***}$-sets (and therefore for $B_k^{**}$-sets as well). We also give a corresponding supersaturation result, which states that above this threshold, the number of additively generic solutions is always at least a constant fraction of the number of such solutions in a binomial random subset of $[N]$ of the same density.

\begin{theorem}\label{thm:additively generic integers}
    For every integer $k \geq 2$, there exist constants $C,c > 0$ such that for every positive integer $N$ and every subset $A \subseteq [N]$ of size at least $CN^{1/k}$, there are at least $c|A|^{2k}/N$ additively generic solutions $(x_1, \ldots, x_{2k}) \in A^{2k}$ to the equation~\eqref{eq:equal k-sums}.
    In particular, a $B_k^{***}$-set in $[N]$ has size at most~$CN^{1/k}$.
\end{theorem}

We also prove an analogous result for finite Abelian groups. Here, we distinguish between groups of even and odd order, as the definition of a $B_k^{***}$-set begins to trivialize as we introduce elements of order $2$. For instance, in the group $\Gamma = \mathbb F_2^n$, a solution to~\eqref{eq:equal k-sums} for $k \geq 2$ cannot be additively generic because the subset sums of any set form a linear subspace of $\mathbb F_2^n$, so there cannot be exactly $2^{2k}-1$ of them. For this reason, in Abelian groups of even order, the best we can hope for in general is to find inseparable solutions to~\eqref{eq:equal k-sums}.

\begin{theorem}\label{thm:inseparable and additively generic finite group}
    For every integer $k \geq 2$, there exist constants $C, c > 0$ such that for every finite Abelian group $\Gamma$ of order $N$ and every subset $A \subseteq \Gamma$ of size at least $CN^{1/k}$, there are at least $c|A|^{2k}/N$ inseparable solutions $(x_1, \ldots, x_{2k}) \in A^{2k}$ to the equation~\eqref{eq:equal k-sums}.
    In particular, a $B_k^{**}$-set in $\Gamma$ has size at most $CN^{1/k}$.
    
    Moreover, if the order of $\Gamma$ is odd, then these solutions can be chosen to be additively generic. In particular, a $B_k^{***}$-set in $\Gamma$ has size at most $CN^{1/k}$.
\end{theorem}
\noindent
\textbf{Remark.} Our proof method does not actually require that the group have odd order in order to give the same conclusion for $B_k^{***}$-sets. All that is needed for the proof to go through is that the group have a bounded number of elements of order $2$. In particular, the maximum size of a $B_k^{***}$-set in the cyclic group $\mathbb Z_N$ is $O(N^{1/k})$ for any $N$, and the corresponding supersaturation result for additively generic solutions holds.

\subsubsection{Our methods: revisiting \texorpdfstring{Bondy--Simonovits}{Bondy-Simonovits}}\label{sec:intro graphs}
Our proofs of Theorems~\ref{thm:additively generic integers} and~\ref{thm:inseparable and additively generic finite group} follow a general strategy that can be employed in several other settings outside additive combinatorics. As a toy example, we use our method to give a new proof of the Bondy--Simonovits Theorem~\cite{bondy1974cycles} for graphs (discussed below), and we later give analogues of this theorem in vector space and matroid settings. We now discuss more of the background of the Bondy--Simonovits Theorem, and in \Cref{sec:graphs}, we outline our method in detail.

The problem of finding nondegenerate solutions to~\eqref{eq:equal k-sums} is reminiscent of the problem in extremal graph theory of finding cycles of a fixed even length $2k$ in graphs.
Similar to the simple argument outlined above for bounding the size of $B_k$-sets in $[N]$, it is not hard to see that any $n$-vertex graph with average degree at least $(2 + o(1))n^{1/k}$ contains two distinct paths of length $k$ between two vertices, and therefore contains a cycle of length \emph{at most} $2k$. Guaranteeing a cycle of length \emph{exactly} $2k$ at the same average degree threshold of~$\Theta(n^{1/k})$ is far less trivial. This was first accomplished by Erd\H{o}s (unpublished), and independently by Bondy and Simonovits~\cite{bondy1974cycles}, and their result has come to be known as the Bondy--Simonovits Theorem. 
\begin{theorem}[Erd\H{o}s; Bondy, Simonovits~\cite{bondy1974cycles}]\label{thm:bondy-simonovits}
    For every integer $k \geq 2$, there exists a constant $C > 0$ such that every $n$-vertex graph with at least $Cn^{1+1/k}$ edges contains a cycle of length $2k$.
\end{theorem}

At the moment, this bound is only known to be tight when $k \in \{2,3,5\}$ \cite{benson1966minimal, lazebnik1995explicit, mellinger2005constructions, wenger1991extremal}. For other values of~$k$, it is a major open problem to determine the existence of $C_{2k}$-free graphs with $\Omega(n^{1+1/k})$ edges, where~$C_{2k}$ denotes the cycle on $2k$ vertices.

The Bondy--Simonovits Theorem is a special case of the \emph{Tur\'an problem} in extremal graph theory. The \emph{extremal number} (or \emph{Tur\'an number}) of a graph $H$, denoted $\ex(n,H)$, is the maximum number of edges in an $n$-vertex graph excluding $H$ as a subgraph. In this language, the Bondy--Simonovits Theorem (\Cref{thm:bondy-simonovits}) asserts that $\ex(n,C_{2k}) = O(n^{1+1/k})$ for every fixed $k \geq 2$.

The \emph{supersaturation problem} is then to guarantee not just one copy of $H$ above a certain threshold, but many copies of $H$. In~\cite{erdos1984cube}, Erd\H{o}s and Simonovits conjectured the following.
\begin{conjecture}[Erd\H{o}s, Simonovits~\cite{erdos1984cube}]\label{conj:erdos simonovits}
    For every bipartite graph $H$, there exist constants $\delta, C,c > 0$ such that every $n$-vertex graph $G$ with $\alpha n^2/2 \geq Cn^{2-\delta}$ edges contains at least $c\alpha^{e(H)}n^{v(H)}$ copies of $H$.
\end{conjecture}
Perhaps non-obviously, this is equivalent to a conjecture of Sidorenko~\cite{S93,S91} which states that for any bipartite graph $H$, constant $\alpha > 0$, and $n$-vertex graph $G$ with $\alpha n^2/2$ edges, there are at least $\alpha^{e(H)}n^{v(H)}$ graph homomorphisms from $H$ to $G$. Both of these counts are asymptotically best possible if true by considering the case where $G$ is a binomial random graph with edge density $\alpha$. Sidorenko's Conjecture has been extensively studied. For recent progress, see~\cite{conlon2010approximate,conlon2018advances,
conlon2017reflection,conlon2021blowups,conlon2024domination,
conlon2024extremal,coregliano2024left,coregliano2021biregularity,
hatami2010graph,im2026sidorenko,kim2016approaches,li2025logarithmic,
szegedy2014information}. In~\cite{erdos1984cube}, Erd\H{o}s and Simonovits showed interest in optimizing the value of $\delta$, and even the value of $C$, for which \Cref{conj:erdos simonovits} holds, and much of the subsequent work on their conjecture has been with this objective. In particular, it may be the case that for every bipartite graph $H$ with $\ex(n,H) = O(n^{2-\delta})$, there exist constants $C,c > 0$ such that every $n$-vertex graph with $\alpha n^2/2 \geq Cn^{2-\delta}$ edges contains at least $c\alpha^{e(H)}n^{v(H)}$ copies of $H$. Simonovits (see~\cite{erdos1984cube}) verified this stronger assertion for cycles of length $2k$ and $\delta = 1-1/k$ by proving the following strengthening of the Bondy--Simonovits Theorem.

\begin{theorem}[Simonovits (see~\cite{erdos1984cube})]\label{thm:simonovits cycle supersaturation}
For every integer $k \geq 2$, there exist constants $C,c > 0$ such that every $n$-vertex graph with at least $Cn^{1+1/k}$ edges contains at least $c \cdot e(G)^{2k}/n^{2k}$ cycles of length $2k$.
\end{theorem}

Our new proof of the Bondy--Simonovits Theorem actually recovers the stronger \Cref{thm:simonovits cycle supersaturation}, and this proof can be found in \Cref{sec:bondy proof}. Other proofs of the Bondy--Simonovits Theorem and strengthenings thereof can be found in
\cite{bukh2017bound, bukh2017erratum, pikhurko2012note,verstraete2000arithmetic}.
The corresponding supersaturation result also has been reproved, and generalizations have been established in~\cite{corsten2021balanced, janzer2023rainbow,jiang2020supersaturation, kim2024rainbow, morris2016number}.

\subsection{Systems of linear equations}

In the vector space setting, there is an even stricter notion of nondegeneracy of solutions to systems of equations which takes into account linear structure as well as Abelian group structure.

\begin{definition}
    Let $\mathbb F$ be a field, and let
    \[\label{eq:system}\left \{ \begin{array}{c c c}\lambda_{1,1} x_1 + \cdots + \lambda_{1,k} x_k & = & 0 \\
    & \vdots & \\
    \lambda_{m,1} x_1 + \cdots + \lambda_{m,k} x_k & = & 0
\end{array}\right.\tag{\textbf{LSys}}\]
    be a system of linear equations in variables $x_1, \ldots, x_k$ with coefficients $\lambda_{1,1}, \ldots, \lambda_{m,k} \in \mathbb F$. Given a vector space~$V$ over~$\mathbb F$, we say that a solution $(x_1, \ldots, x_k) \in V^k$ to~\eqref{eq:system} is $\mathbb F$-\emph{linearly generic} if, for every sequence $(\mu_1, \ldots, \mu_k) \in \mathbb F^k$ which is not an $\mathbb F$-linear combination of $(\lambda_{1,1}, \ldots, \lambda_{1,k}), \ldots, (\lambda_{m,1}, \ldots, \lambda_{m,k})$, we have 
    \[\mu_1x_1 + \cdots + \mu_k x_k \neq 0.\]
\end{definition}
When every nonzero equation implied by ~\eqref{eq:system} has at least $3$ nonzero coefficients, then~\eqref{eq:system} is called \emph{simple}.\footnote{Indeed,~\eqref{eq:system} is simple if and only if the corresponding matroid on $[k]$ whose circuits are the inclusion-minimal nonempty supports of $\mathbb F$-linear combinations of $(\lambda_{1,1}, \ldots, \lambda_{1,k}), \ldots, (\lambda_{m,1}, \ldots, \lambda_{m,k})$ is simple. We refer the reader to~\cite{oxley2006matroid} for matroid terminology used here and throughout.} In this case, linearly generic solutions to~\eqref{eq:system} can also be described geometrically, as follows. Let $(x_1, \ldots, x_k) = (f_1, \ldots, f_k) \in W^k$ be a fixed $\mathbb F$-linearly generic solution to~\eqref{eq:system} in some $\mathbb F$-vector space~$W$, and let $F := \{f_1, \ldots, f_k\} \subseteq W$. By shrinking $W$ if necessary, we can assume that $W = \spn(F)$. Then for any other $\mathbb F$-vector space $V$, the solutions $(x_1, \ldots, x_k) \in V^k$ to~\eqref{eq:system} in $V$ are precisely the images of $(f_1, \ldots, f_k)$ under linear maps $\varphi: W \to V$. Moreover, these solutions are $\mathbb F$-linearly generic if and only if $\varphi$ is injective. Thus we can think of linearly generic solutions to~\eqref{eq:system} in $V$ as ordered linear copies of this fixed \emph{linear configuration} $F$ in $V$. In particular, a larger host configuration $A \subseteq V$ contains an $\mathbb F$-linearly generic solution to~\eqref{eq:system} if and only if $A$ contains $F$ as a subconfiguration, up to linear isomorphism. When $\mathbb F = \mathbb F_q$ for some prime power $q$, we write $\ex_q(n,F)$ for the maximum size of a host configuration $A \subseteq \mathbb F_q^n$ with no linear copy of $F$.

In 1985, Furstenberg and Katznelson~\cite{FK85} proved that any subset of $\mathbb F_q^n$ of positive asymptotic density (in the parameter $n$) must contain large affine subspaces, so we know that $\ex_q(n,F) = o(q^n)$ whenever $F$ is a linear configuration which is contained in some affine subspace excluding the origin. Such configurations are called \emph{affine-type}. In terms of equations, we similarly say that the system~\eqref{eq:system} is \emph{affine-type} (or \emph{translation-invariant}), if it satisfies the property that for every $t \in [m]$, we have $\sum_{i=1}^k \lambda_{t,i} = 0$. Furstenberg and Katznelson's result then implies that the maximum size of $A \subseteq \mathbb F_q^n$ with no linearly generic solution to any affine-type system is $o(q^n)$. If a linear configuration $F$ satisfies the stronger condition that $\ex_q(n,F) = O((q^n)^{1-\delta})$ for some $\delta > 0$, then $F$ is called \emph{temperate}. Perhaps the most famous example of a temperate linear configuration is the $3$-term arithmetic progression over $\mathbb F_3$, which is temperate by the Cap Set Theorem of Ellenberg and Gijswijt~\cite{EG17}. The slice rank method that developed in the wake of their result has since been a major tool for demonstrating that a much larger class of linear configurations is temperate~\cite{BQ00, FP17, G21, S23}, including all affine-type linear configurations determined by a single equation, as well as all affine-type linear configurations over $\mathbb F_2$ and $\mathbb F_3$. But still, for $q > 3$, it is unknown if there exist affine-type linear configurations over $\mathbb F_q$ which are not temperate~\cite{G21}.

Some affine-type systems also exhibit a supersaturation phenomenon analogous to the one conjectured for bipartite graphs by Erd\H{o}s and Simonovits~\cite{erdos1984cube} and by Sidorenko~\cite{S93, S91}, as discussed in \Cref{sec:intro graphs}. Again, consider the system~\eqref{eq:system} of $m$ equations in $k$ variables over $\mathbb F_q$. Assume that~\eqref{eq:system} is simple and that none of the $m$ equations are implied by the others. As before, let $F = \{f_1, \ldots, f_k\} \subseteq W$ be a linear configuration such that $(x_1, \ldots, x_k) = (f_1, \ldots, f_k)$ is an $\mathbb F_q$-linearly generic solution to~\eqref{eq:system}. We then say that~\eqref{eq:system} is \emph{Sidorenko} if, for every linear configuration $A \subseteq \mathbb F_q^n$ of size $\alpha N$, the number of solutions to~\eqref{eq:system} in $A$ is at least $\alpha^k N^{k-m}$, where $N := q^n$. This is equivalent to saying that there are at least $\alpha^k N^{k-m}$ \emph{homomorphisms} from $F$ to $A$, where a map $\varphi : F \to A$ is a \emph{homomorphism of linear configurations} if $\varphi$ extends to a linear map $W \to \mathbb F_q^n$. In this case, the configuration $F$ is called \emph{Sidorenko}. Also, for perhaps non-obvious reasons analogous to the graph setting, this property is equivalent to the assertion that there exist constants $\delta, C, c > 0$ such that every linear configuration $A \subseteq \mathbb F_q^n$ of size $\alpha N \geq CN^{1-\delta}$ contains at least~$c \cdot \alpha^k N^{k-m}$ linear copies of $F$. Note that the \emph{rank} of $F$ is $k-m$, meaning that the elements of $F$ span a $(k-m)$-dimensional linear subspace of $W$. This leads us to ask the following question.
\begin{question}\label{ques:supersaturation Fq}
    For which $\mathbb F_q$-linear configurations $F$ and constants $\delta > 0$ satisfying $\ex_q(n,F) = O(N^{1-\delta})$ do there exist constants $C,c > 0$ such that every linear configuration $A \subseteq \mathbb F_q^n$ of size $\alpha N \geq CN^{1-\delta}$ contains at least $c \cdot \alpha^{|F|}N^{\mathrm{rank}(F)}$ linear copies of $F$, where $N := q^n$?
\end{question}

Saad and Wolf~\cite{SW17} were the first to study Sidorenko systems of equations. Among other things, they verified the Sidorenko property for the symmetric equation
\[\lambda_1 x_1 + \cdots + \lambda_k x_k = \lambda_1 x_{k+1} + \cdots + \lambda_k x_{2k}\tag{\ref{eq:symmetric equation}}\]
over any finite field $\mathbb F_q$ for any choice of coefficients $\lambda_1, \ldots, \lambda_k \in \mathbb F_q^*$. It turns out that~\eqref{eq:symmetric equation} is the only single equation with nonzero coefficients which is Sidorenko, as proved by Fox, Pham, and Zhao~\cite{FPZ21}. For systems of one or more equations, there has been much recent progress giving both positive and negative results~\cite{A22a, altman2023common,altman2025uncommonness,cho2024norming,dong2024uncommon,frederickson2025new, FY25,KLM22,KLM23,SW17,V21,V23}, but a full characterization has not been conjectured except over $\mathbb F_2$, where the present authors posed the following.
\begin{conjecture}[Frederickson, Yepremyan~\cite{FY25}]
    Every affine-type linear configuration in $\mathbb F_2^n$ is Sidorenko. Equivalently, every affine-type, simple system of linear equations over $\mathbb F_2$ is Sidorenko.
\end{conjecture}

Analogous questions have also been studied in integer and finite Abelian group settings~\cite{CE23,DMPT24,G20,SW17,V23,W10}.

\subsubsection{Linear configurations and graphs}

There is also a deep connection between Sidorenko's conjecture for graphs and the analogous problem for vector spaces.
To an $r$-partite, $r$-uniform hypergraph with vertex set $V(H) = \{v_1, \ldots, v_k\}$, we may naturally associate a linear configuration $\Vect_q[H] \subseteq \mathbb F_q^{k}$ given by
\[\Vect_q[H] := \left\{\sum_{j=1}^r e_{i_j} : v_{i_1} \cdots v_{i_r}  \in E(H), \: i_1 < \cdots < i_r\right\},\]
    where $e_i$ denotes the $i$-th standard basis vector in $\mathbb F_q^{k}$. In particular, if $H$ is a bipartite graph, then
    \[\Vect_q[H] := \left\{e_i + e_j : v_iv_j  \in E(H),\: i < j\right\}.\]
    Cho, Conlon, Lee, Skokan, and Versteegen~\cite{cho2024norming} observed that the Sidorenko property for a bipartite graph~$H$ implies the Sidorenko property for $\Vect_q[H]$, and the same holds for the so-called \emph{(weak) norming property} of graphs and linear configurations (equivalently, systems of equations). Their argument essentially amounts to considering host graphs which are \emph{Cayley sum graphs} for configurations $A \subseteq \mathbb F_q^n$, and it also generalizes to hypergraphs.
    
\begin{theorem}[Cho, Conlon, Lee, Skokan, Versteegen~\cite{cho2024norming}]
    Let $H$ be a $k$-vertex, $r$-partite, $r$-uniform hypergraph satisfying the Sidorenko property. Then for every prime power $q$, the linear configuration $\Vect_q[H] \subseteq \mathbb F_q^k$ is Sidorenko.
\end{theorem}

It therefore makes sense to consider the following weakening of Sidorenko's conjecture.

\begin{conjecture}
    For every $k$-vertex bipartite graph $H$ and prime power $q$, the linear configuration $\Vect_q[H] \subseteq \mathbb F_q^k$ is Sidorenko.
\end{conjecture}

One may also wonder if bounds on the extremal number of a bipartite graph $H$ similarly imply anything about $\ex_q(n,\Vect_q[H])$. In particular, for which graphs $H$ and constants $\delta > 0$ satisfying $\ex(n,H) = O(n^{2-\delta})$ do we have $\ex_q(n,\Vect_q[H]) = O((q^n)^{1-\delta})$? We give a few examples of this phenomenon for particular $H$ and~$\delta$. The first gives an $\mathbb F_q$-analogue of the K\H{o}v\'ari--S\'os--Tur\'an Theorem~\cite{kovari1954problem}, which states that for every pair of positive integers $s \leq t$, the complete bipartite graph $K_{s,t}$ satisfies $\ex(n,K_{s,t}) = O(n^{2-1/s})$. In fact, it gives an analogue of its hypergraph generalization, due to Erd\H{o}s~\cite{erdos1964extremal}, that every $n$-vertex, $r$-uniform hypergraph with no copy of the $r$-uniform complete $r$-partite hypergraph $K^{(r)}_{s_1, \ldots, s_r}$ with parts of size $s_1 \leq \cdots \leq s_r$ has $O(n^{r-1/(s_1 \cdots s_{r-1})})$ edges, and it establishes a supersaturation result as in \Cref{ques:supersaturation Fq}.

\begin{theorem}\label{thm:kst vector space}
    For any positive integers $s_1 \leq \cdots \leq s_r$ and prime power $q$, there exist constants $C,c > 0$ such that every linear configuration $A \subseteq \mathbb F_q^n$ of size $|A| = \alpha N \geq CN^{1-1/(s_1 \cdots s_{r-1})}$ contains at least $c \cdot \alpha^{|F|}N^{\mathrm{rank}(F)}$ linear copies of $F = \Vect_q[K^{(r)}_{s_1, \ldots, s_r}]$, where $N := q^n$. In particular, $\ex_q(n,\Vect_q[K^{(r)}_{s_1, \ldots, s_r}]) < CN^{1-1/(s_1 \cdots s_{r-1})}$.
\end{theorem}

The bound $\ex_q(n,\Vect_q[K^{(r)}_{s_1, \ldots, s_r}]) = O(N^{1-1/(s_1 \cdots s_{r-1})})$ also follows directly from Lemma~4.1 in~\cite{FY25} (see also Lemma~6.2.1 in~\cite{Fred2026thesis}) by observing that $\Vect_q[K^{(r)}_{s_1, \ldots, s_r}]$ is a product of $r$ linearly independent sets. A slightly weaker version of this theorem in the $r=2$ case was additionally proved in~\cite{ge2026k} for simple $\mathbb F_q$-representable matroids with no restriction isomorphic to the cycle matroid of $K_{s,t}$. Our short, self-contained proof of \Cref{thm:kst vector space} can be found in \Cref{sec:kovari sos turan}.

Our second example motivated by extremal numbers of bipartite graphs gives an $\mathbb F_q$-analogue of a far-reaching generalization of the K\H{o}v\'ari--S\'os--Tur\'an Theorem due to Alon, Krivelevich, and Sudakov~\cite{alon2003turan}, which states that $\ex(n,H) = O(n^{2-1/s})$ holds for any bipartite graph $H$ in which all vertices in one of the two partition classes have degree at most $s$.

\begin{theorem}\label{thm:drc vector space}
    For every prime power $q$ and bipartite graph $H$ with maximum degree $s$ on one side, we have $\ex_q(n,\Vect_q[H]) = O(N^{1-1/s})$, where $N := q^n$.
\end{theorem}

Though our proof does not give supersaturation in the same way that our other proofs do, we include it in order to highlight an application of the dependent random choice technique (see~\cite{fox2011dependent}) in the vector space setting. It is also possible to prove \Cref{thm:drc vector space} using the result of Alon, Krivelevich, and Sudakov~\cite{alon2003turan} as a black box.

\subsubsection{Linearly generic solutions to symmetric equations}

We now return to our discussion of finding nondegenerate solutions to symmetric linear equations. First observe that in $\mathbb F_2^n$, any inseparable solution to~\eqref{eq:equal k-sums} is necessarily $\mathbb F_2$-linearly generic. Similarly, in $\mathbb F_3^n$, any additively generic solution to~\eqref{eq:equal k-sums} is necessarily $\mathbb F_3$-linearly generic. Moreover, over $\mathbb F_2$ and $\mathbb F_3$, any symmetric linear equation of the form~\eqref{eq:symmetric equation} can be rearranged to become~\eqref{eq:equal k-sums}, up to renaming of variables. Thus we have the following corollary to \Cref{thm:inseparable and additively generic finite group}, which in particular establishes that for $q \in \{2,3\}$ and $k \geq 2$, we have $\ex_q(n,\Vect_q[C_{2k}]) = O(N^{1/k})$, where $N := q^n$, with supersaturation in the sense of \Cref{ques:supersaturation Fq}.
\begin{cor}\label{cor:linearly generic 2 and 3}
    For every integer $k \geq 2$ and $q \in \{2,3\}$, there exist constants $C,c > 0$ such that for every linear configuration $A \subseteq \mathbb F_q^n$ of size at least $CN^{1/k}$ and every $\lambda_1, \ldots, \lambda_k \in \mathbb F_q^*$, there exist at least $c|A|^{2k}/N$ many $\mathbb F_q$-linearly generic solutions to
    \[\lambda_1 x_1 + \cdots + \lambda_k x_k = \lambda_1 x_{k+1} + \cdots + \lambda_k x_{2k}, \tag{\ref{eq:symmetric equation}}\]
    where $N := q^n$.
\end{cor}
When $q=2$, the bound $CN^{1/k}$ is tight up to the constant $C$ for all $k \geq 2$, due to a construction of Bose and Ray-Chaudhuri~\cite{bose1960class}, closely related to so-called \emph{BCH codes}.
\begin{theorem}[Bose, Ray-Chaudhuri~\cite{bose1960class}]\label{thm:bch}
    For every pair $r,k$ of positive integers, there exists a linear configuration $A \subseteq \mathbb F_2^{rk}$ of size $2^r - 1$ with the property that every $2k$ vectors in $A$ are linearly independent.
\end{theorem}

When $q = 3$, the bound $CN^{1/k}$ in \Cref{cor:linearly generic 2 and 3} is also tight up to the constant $C$ for $k=2$, as Huang, Tait, and Won~\cite{huang2019sidon} constructed linear configurations $A \subseteq \mathbb F_3^n$ of size $\Theta(3^{n/2})$ such that every set of~$4$ vectors in $A$ is linearly independent (see also~\cite{pavese2025on}).

Over larger finite fields, Mimura and Tokushige~\cite{mimura2021solving} proved that when $q$ is prime, the maximum size of a subset of $\mathbb F_q^n$ with no solution to~\eqref{eq:symmetric equation} with distinct elements is~$O(N^{1/2})$ for any choice of coefficients $\lambda_1, \ldots, \lambda_k \in \mathbb F_q^*$, where $N := q^n$. Using the Sidorenko property for~\eqref{eq:symmetric equation}, one can easily deduce the stronger statement that over any finite field $\mathbb F_q$, the maximum size of a subset of $\mathbb F_q^n$ with no $\mathbb F_q$-linearly generic solution to~\eqref{eq:symmetric equation} is $O(N^{1/2})$ (see, for instance, Lemma 3.2 in~\cite{FY25}). For $k \geq 3$, this is, of course, quite far from the bound of $O(N^{1/k})$ given by \Cref{cor:linearly generic 2 and 3} for $q \in \{2,3\}$. However, for certain choices of field and coefficients, the bound $O(N^{1/2})$ is not far from the truth. In the same paper~\cite{mimura2021solving}, Mimura and Tokushige proved that for any $\delta > 0$, there exist a prime $p$, coefficients $\lambda_1, \ldots, \lambda_3 \in \mathbb F_p^*$, and a sequence of subsets $A \subseteq \mathbb F_p^n$ of size $\Omega(N^{1/2-\delta})$ with no solution to~\eqref{eq:symmetric equation} with distinct elements, where $N := p^n$. When we consider linearly generic solutions and fields of non-prime order, we see that such constructions can be made even larger; for any prime power $q$, we have that $A := \mathbb F_q^n \subseteq \mathbb F_{q^2}^n$ has size $N^{1/2}$ and has no linearly generic solution to any equation of the form~\eqref{eq:symmetric equation} with $\lambda_1 = 1$ and $\lambda_2 \in \mathbb F_{q^2} \setminus \mathbb F_q$, where $N := (q^2)^n$.

Thus, for $k \geq 3$, we see some natural difficulties for finding linearly generic solutions to an arbitrary single equation of the form~\eqref{eq:symmetric equation} in the host configurations of size $\Theta(N^{1/k})$ that we have been considering thus far. Instead, what we prove is that at this density, we can always find many linearly generic solutions to \emph{some} symmetric equation in $2k$ variables.

\begin{theorem}\label{thm:linearly generic}
    For every integer $k \geq 2$ and prime power $q$, there exist constants $C, c > 0$ such that for every linear configuration $A \subseteq \mathbb F_q^n$ of size at least $CN^{1/k}$, there exist nonzero coefficients $\lambda_1, \ldots, \lambda_k \in \mathbb F_q^*$ and at least $c|A|^{2k}/N$ many $\mathbb F_q$-linearly generic solutions $(x_1, \ldots, x_{2k}) \in A^{2k}$ to the equation
    \[\lambda_1x_1 + \cdots + \lambda_k x_k = \lambda_1x_{k+1} + \cdots + \lambda_kx_{2k}, \tag{\ref{eq:symmetric equation}}\]
    where~$N := q^n$.
\end{theorem}

\noindent
\textbf{Remark.} Our proof can be easily adapted to find at least $c|A|^{tk}/N^{t-1}$ linearly generic solutions to a system of the form
\[\lambda_1x_1 + \cdots + \lambda_k x_k = \lambda_1x_{k+1} + \cdots + \lambda_kx_{2k} = \cdots = \lambda_1x_{(t-1)k+1} + \cdots + \lambda_kx_{tk}\]
for any $t \geq 2$ in any linear configuration $A \subseteq \mathbb F_q^n$ of size at least $CN^{1/k}$, where $C,c > 0$ depend on $k$, $t$, and~$q$, and $\lambda_1, \ldots, \lambda_k$ may depend on $A$. Similarly, our proof of \Cref{thm:simonovits cycle supersaturation} can be modified to recover the known result that any $n$-vertex graph with more than $Cn^{1+1/k}$ edges contains at least $c \cdot e(G)^{tk}/n^{tk+t-2}$ copies of the \emph{theta graph} $\theta_{t,k}$, consisting of $t$ internally vertex-disjoint paths of length $k$ between two vertices, where $C,c > 0$ depend on $k$ and $t$. Faudree and Simonovits~\cite{faudree1983class} were the first to prove the bound $\ex(n,\theta_{t,k}) = O(n^{1+1/k})$. A supersaturation version was later given in~\cite{jiang2020supersaturation}, and a rainbow version was proved in~\cite{janzer2023rainbow}.

\subsubsection{An application to matroids}

\Cref{thm:linearly generic} has an especially nice interpretation in terms of matroids. Observe that for any choice of coefficients $\lambda_1, \ldots, \lambda_k \in \mathbb F_q^*$, if $(x_1, \ldots, x_{2k}) \in A^{2k}$ is an $\mathbb F_q$-linearly generic solution to~\eqref{eq:symmetric equation} in some $A \subseteq \mathbb F_q^n$, then the set $D := \{x_1, \ldots, x_{2k}\}$ is linearly dependent, but every proper subset of $D$ is linearly independent. In other words, $D$ corresponds to a \emph{circuit} in the vector matroid of the $(n \times |A|)$ matrix whose columns are the elements of $A$. Thus \Cref{thm:linearly generic} gives an $\mathbb F_q$-representable matroid analogue of the Bondy--Simonovits Theorem, with supersaturation, for every prime power $q$.

\begin{cor}\label{cor:matroid bondy-simonovits}
    For every integer $k \geq 2$ and prime power $q$, there exist constants $C, c > 0$ such that every simple $\mathbb F_q$-representable matroid $M$ of rank $n$ and size $|E(M)| \geq CN^{1/k}$ contains at least $c|E(M)|^{2k}/N$ circuits of size exactly $2k$, where $N := q^n$.
\end{cor}

Recall that when $q=2$, we know by \Cref{thm:bch} that for every $k \geq 2$, there exist rank-$n$ simple binary matroids of size $\Theta(2^{n/k})$ with girth greater than $2k$, making \Cref{cor:matroid bondy-simonovits} tight over $\mathbb F_2$. This is in contrast with the graph setting, where we only know that the Bondy--Simonovits Theorem is tight for $k \in \{2,3,5\}$.
\begin{cor}
    For every $k \geq 2$, the maximum size of a rank-$n$ simple binary matroid with no circuit of size exactly $2k$ is $\Theta(N^{1/k})$, where $N := 2^n$.
\end{cor}

\noindent
\textbf{Notation.} Whenever we are in the context of a specified finite Abelian group $\Gamma$ (or vector space as a special case), we reserve the variable $N$ for the order of $\Gamma$, which we may not always specify. For an Abelian group~$\Gamma$ and a finite subset $U = \{a_1, \ldots, a_k\} \subseteq \Gamma$ with $a_1, \ldots, a_k$ distinct, we define
\begin{align*}
\spn^{\pm}(U) &:= \left\{\sum_{i=1}^k \sigma_i a_i : \sigma_i \in \{-1, 0, 1\} \text{ for all $i \in [k]$}\right\}, \\
    \spn^{\pm 2}(U) &:= \left\{\sum_{i=1}^k \sigma_i a_i : \sigma_i \in \{-2,-1, 0, 1, 2\} \text{ for all $i \in [k]$}\right\}.
\end{align*}
We say that $U$ is \emph{dissociated} if $\sum_{i=1}^k \sigma_i a_i \neq 0$ for every nonzero tuple $(\sigma_1, \ldots, \sigma_k) \in \{-1,0,1\}^k$. Equivalently,~$U$ is dissociated if the $2^{k}$ subset sums $\sum_{i \in S} a_i$ for $S \subseteq [k]$ are distinct. Moreover, we say that $U$ is \emph{strongly dissociated} if $\sum_{i=1}^k \sigma_i a_i \neq 0$ for every nonzero tuple $(\sigma_1, \ldots, \sigma_k) \in \{-2,-1,0,1,2\}^k$. 
\vspace{10pt}

\noindent
\textbf{Organization.} The rest of the paper is organized as follows. In \Cref{sec:prelim}, we record a standard supersaturation technique and give a short proof of \Cref{thm:Bk construction vector space}, which constructs large $B_k$-sets in finite vector spaces. \Cref{sec:kovari sos turan} is self-contained and proves our K\H{o}v\'ari--S\'os-Tur\'an-type theorems for vector spaces (Theorems~\ref{thm:kst vector space} and ~\ref{thm:drc vector space}). In \Cref{sec:graphs}, we outline our main proof method and use it to prove the Bondy--Simonovits Theorem as a toy example. Sections~\ref{sec:inseparable} and~\ref{sec:additively generic} are devoted to proving Theorems~\ref{thm:additively generic integers} and~\ref{thm:inseparable and additively generic finite group}, which find nondegenerate solutions to~\eqref{eq:equal k-sums} in Abelian group settings. In \Cref{sec:linearly generic}, we discuss the straightforward way in which these proofs adapt to the vector space setting in order to prove \Cref{thm:linearly generic}. Finally, we pose several open problems in \Cref{sec:concluding remarks}.

\section{Preliminaries}\label{sec:prelim}

In order to prove some of our supersaturation results, we will use the following standard random sampling argument to convert supersaturation results which hold \emph{at} a certain density threshold into supersaturation results which hold \emph{anywhere above} that same threshold.
\begin{prop}\label{prop:supersaturation lifting hypergraph}
    Let $r \leq m \leq n$, and let $c > 0$. If $\mathcal H$ is a $r$-uniform hypergraph on $n$ vertices in which every set of $m$ vertices contains at least $cm^r$ edges of $\mathcal H$, then $e(\mathcal H) \geq cn^r$.
\end{prop}
\begin{proof}
    Let $S$ be a subset of $V(\mathcal H)$ of size exactly $m$ chosen uniformly at random. By assumption, we have $e(\mathcal H[S]) \geq c m^r$ with probability $1$, and so
    \[\mathbb E[e(\mathcal H[S])] \geq cm^r.\]
    On the other hand, each edge of $\mathcal H$ is contained in $S$ with probability $\binom{m}{r}/\binom{n}{r} \leq m^r/n^r$, so 
    \[\mathbb E[e(\mathcal H[S])] \leq \frac{m^r}{n^r}e(\mathcal H).\]
    Combining these two estimates on $\mathbb E[e(\mathcal H[S])]$ gives the desired count on $e(\mathcal H)$.
\end{proof}
In the graph setting, \Cref{prop:supersaturation lifting hypergraph} implies that if we can show that any $n$-vertex graph with \emph{exactly} $m$ edges contains at least $cm^{e(H)}/n^{2e(H)-v(H)}$ copies of some graph $H$, then it immediately follows that any $n$-vertex graph with \emph{at least} $m$ edges contains at least $c \cdot e(G)^{e(H)}/n^{2e(H)-v(H)}$ copies of $H$. We will also apply \Cref{prop:supersaturation lifting hypergraph} in additive and vector space settings to an analogous effect.

\subsection{Large \texorpdfstring{$B_k$}{B_k}-sets in finite vector spaces}
We now give the proof of \Cref{thm:Bk construction vector space}, which asserts the existence of a $B_k$-set in $\mathbb F_q^n$ of size $\Omega(N^{1/k})$, where $N := q^n$, whenever the characteristic of $\mathbb F_q$ exceeds $k$. For this, we use a simple moment curve construction over a field extension. Note that when the characteristic of $\mathbb F_q$ is at most $k$, we trivially have that the maximum size of a $B_k$-set in $\mathbb F_q^n$ is~$1$ for every~$n$. For $B_k^*$-sets, this triviality vanishes, and in fact, \Cref{thm:bch} constructs $B_k^*$-sets in $\mathbb F_q^n$ of size $\Omega(N^{1/k})$ whenever $q$ is even. It would be interesting to see if there exist $B_k^*$-sets in $\mathbb F_q^n$ of size $\Omega(N^{1/k})$ for odd prime powers $q \leq k$.
\begin{proof}[Proof of \Cref{thm:Bk construction vector space}]
    Let $p > k$ be prime. If $r=0$, then the statement is trivial, so assume $r \geq 1$. The $\mathbb F_{p^r}$-vector space $\mathbb F_{p^r}^k$ is also an $\mathbb F_p$-vector space, which we can identify with $\mathbb F_p^{rk}$. Define
    \[A := \{(\alpha , \alpha^2 , \ldots , \alpha^k) : \alpha \in \mathbb F_{p^r}\} \subseteq \mathbb F_{p^r}^k \cong \mathbb F_p^{rk}.\]
    Note that $A$ has size exactly $p^r$. To see that $A$ is a $B_k$-set, suppose we have $\alpha_1, \ldots, \alpha_{2k} \in \mathbb F_{p^r}$ satisfying 
    \[\sum_{i=1}^k \alpha_i^\ell = \sum_{i=1}^k \alpha_{k+i}^\ell\]
    for every $1 \leq \ell \leq k$. Because $p > k$, it follows from Newton's identities that the coefficients of the polynomials $(x-\alpha_1) \cdots (x - \alpha_k)$ and $(x - \alpha_{k+1})\cdots (x-\alpha_{2k})$ in $\mathbb F_{p^r}[x]$ coincide, and thus $(\alpha_{k+1}, \ldots, \alpha_{2k})$ is a permutation of $(\alpha_1, \ldots, \alpha_k)$.

    Now if $\mathbb F_q$ has characteristic $p$, then given a positive integer $n$, choose $r \geq 0$ maximum such that $p^{rk} \leq N := q^n$. Since $\mathbb F_p^{rk}$ as an Abelian group is isomorphic to a subgroup of $\mathbb F_q^n$, we see by the previous paragraph that $\mathbb F_q^n$ contains a $B_k$-set of size $p^r > N^{1/k}/p$.
\end{proof}

\section{General bounds for excluded \texorpdfstring{$r$-partite}{r-partite} graphic linear configurations}\label{sec:kovari sos turan}

Next, we will give two short proofs of the $\mathbb F_q$-analogue of the K\H{o}v\'ari--S\'os--Tur\'an Theorem: namely 
\[\ex_q(n,\Vect_q[K_{s,t}]) = O(N^{1-1/s}),\] where $N := q^n$. The first proof establishes supersaturation of copies of $\Vect_q[K_{s,t}]$ and generalizes to $r$-uniform, complete $r$-partite hypergraphs, while the second one extends the extremal bound $O(N^{1-1/s})$ to a larger class of linear configurations arising from bipartite graphs. We begin with the hypergraph version with supersaturation, namely \Cref{thm:kst vector space}. For counting purposes, it will be convenient to define the following auxiliary object which witnesses a linear copy of $\Vect_q[K^{(r)}_{s_1, \ldots, s_r}]$.

\begin{definition}
    Let $A \subseteq \mathbb F_q^n$. A \emph{model} for $\Vect_q[K^{(r)}_{s_1, \ldots, s_r}]$ in $A$ is a linearly independent indexed tuple $(u^{(i)}_j)_{i,j} \in (\mathbb F_q^n)^{s_1 + \cdots + s_r}$, where the indices run over all pairs $(i,j)$ with $1 \leq i \leq r$ and $1 \leq j \leq s_i$, such that $u^{(1)}_{j_1} + \cdots + u^{(r)}_{j_r} \in A$ for all $(j_1, \ldots, j_r) \in [s_1] \times \dots \times [s_r]$. The \emph{content} of the model $(u^{(i)}_j)_{i,j}$ is the set 
    \[\left\{u^{(1)}_{j_1} + \cdots + u^{(r)}_{j_r} : (j_1, \ldots, j_r) \in [s_1] \times \cdots \times [s_r]\right\} \subseteq A,\]
    which is a linear copy of $\Vect_q[K^{(r)}_{s_1, \ldots, s_r}]$ in $A$.
\end{definition}
We now prove a counting lemma for models of $\Vect_q[K^{(r)}_{s_1, \ldots, s_r}]$, which for the purposes of induction, places an additional technical requirement on these models: that they be linearly independent modulo some subspace $W$ of constant dimension, which we remind the reader means that no nontrivial linear combination of these vectors is in $W$.
\begin{lemma}\label{lem:many models}
    For any integer $d \geq 0$, positive integers $r, s_1, \ldots, s_r$, and prime power $q$, there exist constants $C, c > 0$ such that for every linear configuration $A \subseteq \mathbb F_q^n$ of size $|A| = \alpha N \geq CN^{1-1/(s_1 \cdots s_{r-1})}$ and every $d$-dimensional subspace $W \subseteq \mathbb F_q^n$, there are at least $ c\alpha^{s_1 \cdots s_r}N^{s_1 + \cdots + s_r}$ models for $\Vect_q[K^{(r)}_{s_1, \ldots, s_r}]$ in $A$ which are linearly independent modulo $W$, where $N := q^n$.
\end{lemma}
\begin{proof}
    We will prove the lemma by induction on $r$. By choosing a larger constant $C$ if necessary, we can assume that $n$ is sufficiently large. If $r=1$, then we are counting $s_1$-tuples $(u_1, \ldots, u_{s_1}) \in A^{s_1}$ which are linearly independent modulo $W$, of which there are at least
    \[(|A|-q^d)(|A|-q^{d+1})\cdots (|A|-q^{d+s_1-1}) \geq \prod_{i=1}^{s_1}\left(|A| - \frac{|A|}{2^{i}}\right) \geq \frac{|A|^{s_1}}{10} = \frac{1}{10} \alpha^{s_1}N^{s_1}\]
    when $|A| \geq C \geq q^{s_1 + d}$, as required.

    Now assume that $r \geq 2$. For a sequence $\bar u = (u_1, \ldots, u_{s_1}) \in (\mathbb F_q^n)^{s_1}$, let
    \[A_{\bar u} := \bigcap_{j=1}^{s_1} (A - u_j) \subseteq \mathbb F_q^n\]
    denote the set of vectors $v \in \mathbb F_q^n$ such that $u_j + v \in A$ for all $j \in [s_1]$. Let $C',c' > 0$ be the constants given by the inductive hypothesis with the parameters $r,C, c, d$ replaced by $r-1,C', c', d + s_1$, respectively. Let $m_0 := C'N^{1-1/(s_2 \cdots s_{r-1})}$. If $\bar u$ is linearly independent modulo~$W$ and $|A_{\bar u}| \geq m_0$, then there are at least $c' (|A_{\bar u}|/N)^{s_2 \cdots s_r}N^{s_2 + \cdots + s_r}$ models for $\Vect_q[K^{(r-1)}_{s_2, \ldots, s_r}]$ in $A_{\bar u}$ which are linearly independent modulo $\spn(W \cup \{u_1, \ldots, u_{s_1}\})$. For each such model $(u^{(i)}_j)_{i,j}$, with indices running over pairs $(i,j)$ with $2 \leq i \leq r$ and $1 \leq j \leq s_i$, we have for every $(j_1, \ldots, j_r) \in [s_1] \times \cdots \times [s_r]$ that $u^{(2)}_{j_2} + \cdots + u^{(r)}_{j_r} \in A_{\bar u}$, and hence $u_{j_1} + u^{(2)}_{j_2} + \cdots + u^{(r)}_{j_r} \in A$. Thus, setting $u^{(1)}_j = u_j$ for $j \in [s_1]$, we have that $(u^{(i)}_j)_{i,j}$, with $i \in [r]$ and $j \in [s_i]$, is a model for $\Vect_q[K^{(r)}_{s_1, \ldots, s_r}]$ in $A$ which is linearly independent modulo~$W$. In summary, for any $\bar u = (u_1, \ldots, u_{s_1}) \in (\mathbb F_q^n)^{s_1}$ which is linearly independent modulo~$W$, with $|A_{\bar u}| \geq m_0$, the number of models $(u^{(i)}_j)_{i,j}$ for $\Vect_q[K^{(r)}_{s_1, \ldots, s_r}]$ in $A$ which are linearly independent modulo~$W$, with $u^{(1)}_j = u_j$ for all $j \in [s_1]$, is at least
    \begin{equation}\label{eq:initial model count}
        c' (|A_{\bar u}|/N)^{s_2 \cdots s_r}N^{s_2 + \cdots + s_r}.
    \end{equation}

    In order to get a lower bound on the total number of models for $\Vect_q[K^{(r)}_{s_1, \ldots, s_r}]$ we find this way, we define some notation. Let $\mathrm{ind}_W(\mathbb F_q^n \times A^{s_1})$ denote the set of linearly independent sequences (mod $W$) of the form $(v, a_1, \ldots, a_{s_1}) \in \mathbb F_q^n \times A^{s_1}$, which has size at least 
    \[(|A|-q^d)(|A|-q^{d+1})\cdots (|A|-q^{d+s_1-1})(N-q^{d+s_1}) \geq \prod_{i=1}^{s_1}\left(|A| - \frac{|A|}{2^{i}}\right) \cdot \frac{N}{2} \geq \frac{N|A|^{s_1}}{10}\]
    since $|A| \geq CN^{1-1/(s_1 \cdots s_{r-1})}$ and $n$ is sufficiently large.
    Consider the function $f : \mathrm{ind}_W(\mathbb F_q^n \times A^{s_1}) \to (\mathbb F_q^n)^{s_1}$ given by $f(v, a_1, \ldots, a_{s_1}) = (a_1 - v, \ldots, a_{s_1} - v)$. We define the \emph{weight} of a sequence $\bar u = (u_1, \ldots, u_{s_1}) \in (\mathbb F_q^n)^{s_1}$ to be $w(\bar u) := |f^{-1}(\bar u)|$. Note that the average weight $w_{\mathrm{avg}}$ of a sequence $\bar u \in (\mathbb F_q^n)^{s_1}$ is at least 
    \begin{equation}\label{eq:avg weight KST}
        w_{\mathrm{avg}} := \frac{\left|\mathrm{ind}_W(\mathbb F_q^n \times A^{s_1})\right|}{\left|(\mathbb F_q^n)^{s_1}\right|} \geq \frac{|A|^{s_1}}{10N^{{s_1}-1}} \geq \frac{C^{s_1}N^{1-1/(s_2 \cdots s_{r-1})}}{10} \geq 2m_0
    \end{equation}
    if we choose $C$ large enough as a function of $C'$. Also note that for any $\bar u \in (\mathbb F_q^n)^{s_1}$, we have $|A_{\bar u}| \geq w(\bar u)$. Indeed, for fixed~$\bar u$, any sequence $(v, a_1, \ldots, a_{s_1}) \in f^{-1}(\bar u)$ is determined by $v$ since $a_j = u_j + v$ for all $j \in [s_1]$, and such a sequence has $v \in A_{\bar u}$.
    
    Finally, observe that whenever $w(\bar u) > 0$, it follows that $\bar u$ is linearly independent modulo $W$ since $\bar u$ is then of the form $(a_1 - v, \ldots, a_{s_1} - v)$ for $v, a_1, \ldots, a_{s_1}$ which are linearly independent modulo $W$. This, combined with~\eqref{eq:initial model count}, means that for every $\bar u = (u_1, \ldots, u_{s_1}) \in (\mathbb F_q^n)^{s_1}$, there exist at least $g(w(\bar u))$ models $(u^{(i)}_j)_{i,j}$ for $\Vect_q[K^{(r)}_{s_1, \ldots, s_r}]$ in $A$ which are linearly modulo $W$, with $u^{(1)}_j = u_j$ for all $j \in [s_1]$, where
    \[g(x) := \left\{\begin{array}{c l}
    0 & \text{if $x < m_0$,} \\
    c'\left(x/N\right)^{s_2 \cdots s_{r}}N^{s_2 + \cdots + s_r} - c'\left(m_0/N\right)^{s_2 \cdots s_{r}}N^{s_2 + \cdots + s_r} & \text{if $x \geq m_0$.}
    \end{array}\right.\]
    Therefore, since $g$ is a convex and monotone function, we have by Jensen's inequality and by~\eqref{eq:avg weight KST} that the total number of models for $\Vect_q[K^{(r)}_{s_1, \ldots, s_r}]$ in $A$ which are linearly modulo $W$ is at least
    \begin{align*}
        N^{s_1}g(w_{\mathrm{avg}})
        &= N^{s_1} \left(c'\left(w_{\mathrm{avg}}/N\right)^{s_2 \cdots s_{r}}N^{s_2 + \cdots + s_r} - c'\left(m_0/N\right)^{s_2 \cdots s_{r}}N^{s_2 + \cdots + s_r}\right) \\
        &\geq \frac{1}{2}N^{s_1}\left(c'\left(\frac{|A|^{s_1}}{10N^{s_1}}\right)^{s_2 \cdots s_{r}}N^{s_2 + \cdots + s_r}\right) \\
        &= \frac{c'}{2 \cdot 10^{s_2 \cdots s_r}}\alpha^{s_1 \cdots s_r}N^{s_1 + \cdots + s_r}. \qedhere
    \end{align*}
\end{proof}
We now complete the proof of \Cref{thm:kst vector space} by converting this count of models for $\Vect_q[K^{(r)}_{s_1, \ldots, s_r}]$ into a count of linear copies of $\Vect_q[K^{(r)}_{s_1, \ldots, s_r}]$.

\begin{proof}[Proof of \Cref{thm:kst vector space}]
Let $C, c' > 0$ be as in \Cref{lem:many models} with $d, c$ replaced with $0, c'$, respectively, and let $A \subseteq \mathbb F_q^n$ be of size $|A| = \alpha N \geq CN^{1-1/(s_1 \cdots s_{r-1})}$. Then by \Cref{lem:many models}, there are at least $c'\alpha^{s_1 \cdots s_r}N^{s_1 + \cdots + s_r}$ models for $\Vect_q[K^{(r)}_{s_1, \ldots, s_r}]$ in $A$. This overcounts the number of linear copies of $\Vect_q[K^{(r)}_{s_1, \ldots, s_r}]$ in~$A$, as several models may have the same content. However, if $B$ is a linear copy of $\Vect_q[K^{(r)}_{s_1, \ldots, s_r}]$ in~$A$, then we claim that at most $q^{(s_1 + \cdots + s_r)^2}N^{(s_1 + \cdots + s_r) - \mathrm{rank}(B)}$ models for $\Vect_q[K^{(r)}_{s_1, \ldots, s_r}]$ in~$A$ have content $B$. Indeed, if $(u^{(i)}_j)_{i,j}$ is a model for $\Vect_q[K^{(r)}_{s_1, \ldots, s_r}]$ in $A$, then the vectors $u^{(i)}_j$ form a basis of some $(s_1 + \cdots + s_r)$-dimensional space~$W$ with $\spn(B) \subseteq W \subseteq \mathbb F_q^n$. There are at most $N^{(s_1 + \cdots + s_r)-\mathrm{rank}(B)}$ choices for $W$, and given $W$, there are at most $q^{(s_1 + \cdots + s_r)^2}$ choices for an ordered basis of $W$, giving the desired bound. Altogether, this means that the number of linear copies of $F := \Vect_q[K^{(r)}_{s_1, \ldots, s_r}]$ in $A$ is at least
\[\frac{c'\alpha^{s_1 \cdots s_r}N^{s_1 + \cdots +s_r}}{q^{(s_1 + \cdots + s_r)^2}N^{(s_1 + \cdots + s_r) - \mathrm{rank}(F)}} = \frac{c'}{q^{(s_1 + \cdots + s_r)^2}}\alpha^{|F|}N^{\mathrm{rank}(F)},\]
as desired.
\end{proof}

Next, inspired by the analogous theorem for graphs proven by Alon, Krivelevich, and Sudakov~\cite{alon2003turan}, we prove using the dependent random choice method that the K\H{o}v\'ari--S\'os--Tur\'an bound $\ex_q(n,\Vect_q[H]) = O(N^{1-1/s})$ holds for any bipartite graph $H$ with maximum degree at most $s$ on one side.

\begin{proof}[Proof of \Cref{thm:drc vector space}]
    We may assume that $s \leq v(H)$. Choose $C := 2q^{v(H)}$, and let $A \subseteq \mathbb F_q^n$ be a linear configuration of size $|A| \geq CN^{1-1/s}$. Write $\alpha := |A|/N \geq CN^{-1/s}$, and let $b := q^{v(H)}$. Let $(L,R)$ be a bipartition of $H$ such that every vertex in $R$ has degree at most $s$. Write $L = \{y_1, \ldots, y_\ell\}$ and $R = \{z_1, \ldots, z_r\}$. We wish to find linearly independent vectors $u_1, \ldots, u_\ell, v_1, \ldots, v_r \in \mathbb F_q^n$ such that $u_i + v_j \in A$ for every pair $i \in [\ell]$, $j \in [r]$ with $y_iz_j \in E(H)$.
    
    To accomplish this, we will generate a random set $U \subseteq \mathbb F_q^n$ as follows. Choose $\tilde v_1, \ldots, \tilde v_s \in \mathbb F_q^n$ independently and uniformly at random, and define 
    \[U := \bigcap_{i=1}^s(A - \tilde v_i).\]
    Note that every vector $u \in \mathbb F_q^n$ has probability exactly $(|A|/N)^s = \alpha^s$ of being in $U$, so 
    \begin{equation}\label{eq:expected intersection}
        \mathbb E(|U|) = \alpha^s N.
    \end{equation}
    
    Call an $s$-element subset $S \subseteq \mathbb F_q^n$ \emph{bad} if there are at most $b$ elements $v \in \mathbb F_q^n$ such that $u + v \in A$ for all $u \in S$, or equivalently, such that $S \subseteq A - v$. Because of how we selected the set $U$, we do not expect many bad subsets to be contained in $U$. More precisely, if $S$ is bad, then the probability that $S \subseteq U$ is at most $(b/N)^s$, so the expectation of the number $B_U$ of bad $s$-element subsets of $U$ is at most
    \[\mathbb E(B_U) \leq \binom{N}{s}\left(\frac{b}{N}\right)^s \leq \frac{b^s}{s!}.\]
    By the linearity of expectation, we have
    \[\mathbb E(|U|- B_U) \geq \alpha^sN - \frac{b^s}{s!} \geq C^s - \frac{b^s}{s!} \geq q^{v(H)}.\]
    
    Therefore, there exists a set $U \subseteq \mathbb F_q^n$ for which $|U| - B_U \geq q^{v(H)}$. By deleting from $U$ one element from every bad $s$-element subset, we obtain a set $U' \subseteq U$ of size at least $q^{v(H)} \geq s$ such that no $s$-element subset of $U'$ is bad. Since $\ell \leq v(H)$, we may find $\ell$ linearly independent elements $u_1, \ldots, u_\ell \in U'$. For each $j \in [r]$, define $S_j := \{u_i : y_iz_j \in E(H)\} \subseteq U'$, which has size at most $s$ since $z_j \in R$ has degree at most $s$ in $H$ by assumption. Consequently, since $S_j$ is contained in an $s$-element subset of $U'$ which is not bad, there are more than $b = q^{v(H)}$ elements $v \in \mathbb F_q^n$ such that $u + v \in A$ for every $u \in S_j$. Therefore, for each $j \in [r]$, in turn, we may find $v_j$ outside of $\spn\{u_1, \ldots, u_\ell, v_1, \ldots, v_{j-1}\}$ such that $u + v_j \in A$ for every $u \in S_j$; that is, $u_i + v_j \in A$ whenever $y_iz_j \in E(H)$. By design, $\{u_1, \ldots, u_\ell, v_1, \ldots, v_r\}$ is linearly independent, and so $\{u_i + v_j : y_iz_j \in E(H)\}$ is a linear copy of $\Vect_q[H]$ in $A$.
\end{proof}

\section{Even cycles in graphs}\label{sec:graphs}
\subsection{Proof outline}

We will now outline our main method for proving Theorems~\ref{thm:additively generic integers},~\ref{thm:inseparable and additively generic finite group}, and~\ref{thm:linearly generic}. We do this in the graph setting and use it to give a new proof of the Bondy--Simonovits Theorem (\Cref{thm:bondy-simonovits}) and its supersaturation version (\Cref{thm:simonovits cycle supersaturation}).

Ultimately, we want to find cycles of length $2k$ in $n$-vertex undirected graphs with average degree at least $Cn^{1/k}$. However, it will be convenient for us to consider directed graphs instead. Note that an undirected graph can be viewed as a symmetric digraph by replacing every edge by two antiparallel directed edges, and since every undirected graph of average degree $d$ has a subgraph of minimum degree greater than $d/2$, it follows that the digraph obtained by replacing every edge with a directed $2$-cycle has a subgraph which is \emph{$d'$-out-regular} for some $d' > d/2$; i.e., every vertex has out-degree exactly $d'$. 

Therefore, it suffices to consider an $n$-vertex, $d$-out-regular digraph $D$, with $d \geq Cn^{1/k}$. We will be happy if we are able to find \emph{some} orientation of a cycle of length $2k$ which is abundant in $D$. The particular orientations we will be looking for take the following form.
\begin{definition}
    Let $\ell, k$ be positive integers with $2 \leq \ell \leq k$. A \emph{$(k,\ell)$-tulip} is a directed graph consisting of a \emph{root} vertex $u$, vertices $z_0, \ldots, z_{k-\ell}$, $y_1, \ldots, y_{k-\ell}$ which are pairwise distinct and distinct from $u$, edges $y_iz_{i-1}$ and $y_iz_i$ for all $1 \leq i \leq k-\ell$, and two directed paths of length $\ell$: one path $P$ from $u$ to $z_0$, and one path $P'$ from $u$ to $z_{k-\ell}$, which are internally vertex-disjoint from each other and from all named vertices.
\end{definition}
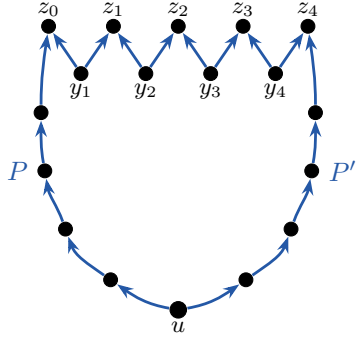
\begin{figure}[h]
\centering
% Scale the coordinates without shrinking the labels.
\pgfmathsetlengthmacro{\tulipunit}{0.25*\textwidth/8.3}
\begin{tikzpicture}[
  x=\tulipunit,
  y=\tulipunit,
  >={Stealth[length=2.15mm,width=1.55mm]},
  vertex/.style={circle,draw=black,fill=black,inner sep=0pt,
                 minimum size=5.2pt,outer sep=.4pt},
  vertex label/.style={font=\small,inner sep=1.2pt,text=black},
  selected edge/.style={->,draw=pathblue,line width=1.05pt},
  path label/.style={font=\small\bfseries,inner sep=1.2pt,text=pathblue}
]

% Root and the internal vertices of the two length-5 paths.
\node[vertex,minimum size=6.2pt] (u) at (0,-.70) {};
\foreach \i/\xx/\yy in {
  1/1.80/.10,
  2/3.00/1.45,
  3/3.55/3.00,
  4/3.65/4.55
}{
  \node[vertex] (l\i) at (-\xx,\yy) {};
  \node[vertex] (r\i) at ( \xx,\yy) {};
}

% The alternating upper part of the tulip.
\foreach \i/\xx in {0/-3.45,1/-1.725,2/0,3/1.725,4/3.45}{
  \node[vertex] (z\i) at (\xx,6.85) {};
}
\foreach \i/\xx in {1/-2.5875,2/-.8625,3/.8625,4/2.5875}{
  \node[vertex] (y\i) at (\xx,5.59) {};
}

% Rounded directed path P.
\draw[selected edge] (u)  to[out=170,in=-35] (l1);
\draw[selected edge] (l1) to[out=145,in=-70] (l2);
\draw[selected edge] (l2) to[out=110,in=-80] (l3);
\draw[selected edge] (l3) to[out=100,in=-90] (l4);
\draw[selected edge] (l4) to[out=90,in=-98] (z0);

% Rounded directed path P'.
\draw[selected edge] (u)  to[out=10,in=215] (r1);
\draw[selected edge] (r1) to[out=35,in=250] (r2);
\draw[selected edge] (r2) to[out=70,in=260] (r3);
\draw[selected edge] (r3) to[out=80,in=270] (r4);
\draw[selected edge] (r4) to[out=90,in=278] (z4);

% Each y_i sends an edge to z_{i-1} and z_i.
\foreach \i/\j in {1/0,2/1,3/2,4/3}{
  \draw[selected edge] (y\i) -- (z\j);
  \draw[selected edge] (y\i) -- (z\i);
}

% Labels, matching the companion figure.
\foreach \i in {0,...,4}{
  \node[vertex label,above=2.5pt] at (z\i) {$z_{\i}$};
}
\foreach \i in {1,...,4}{
  \node[vertex label,below=2.5pt] at (y\i) {$y_{\i}$};
}
\node[path label,left=5pt,font=\bfseries] at (l3) {$P$};
\node[path label,right=5pt,font=\bfseries] at (r3) {$P'$};
\node[vertex label,below=3pt,font=\small\bfseries] at (u) {$u$};

\end{tikzpicture}
\caption{A $(k,\ell)$-tulip with $k=9$ and $\ell=5$.}
\label{fig:tulip}
\end{figure}

For now, suppose we wish to find just one such subgraph in $D$. The first thing to notice is that from every vertex $u$, there are $d^k \geq C^k n$ (directed) walks of length $k$ starting at $u$. Since the digraph has only $n$ vertices, this means that for a typical pair of vertices $(u,z) \in V(D)^2$, there are at least $C^k$ walks of length $k$ from $u$ to $z$. We will want to keep track of the number of walks of various lengths between pairs of vertices, so we make the following definition.

\begin{definition}
    Let $D$ be a fixed directed graph, and let $\ell$ be a positive integer. The \emph{$\ell$-weight} of a pair of vertices $(u,z) \in V(D)^2$, denoted $w_\ell(u,z)$, is the number of walks of length $\ell$ from $u$ to $z$. If $u \in V(D)$ and $Z \subseteq V(D)$, then we define the \emph{$\ell$-weight} of $Z$ with respect to $u$ as
    \[w_{\ell}(u,Z) := \sum_{z \in Z} w_\ell(u,z).\]
\end{definition}

In this terminology, it is easy to find a pair $(u_0,z_0) \in V(D)^2$ with $k$-weight at least $C^k$, and we can hope that two of these walks are internally vertex-disjoint paths. If so, then we have found a $(k,k)$-tulip in $D$. If not, then it must either be the case that a positive proportion of the walks must pass through the same intermediate vertex, or that a positive proportion of the walks are degenerate (i.e., are not paths). Either case witnesses a large $\ell$-weight of some pair $(u_1,z_1) \in V(D)^2$ for some $\ell < k$. If this failure happens for every pair $(u_0,z_0)$ of large $k$-weight, then a careful counting argument shows that there exist $\ell < k$, a vertex $u$, and a small set $Z \subseteq V(D)$ whose $\ell$-weight with respect to $u$ is $\Omega(d^\ell)$. We will consider $\ell$ as small as possible with this property.

For the sake of illustration, suppose that this happens for $\ell = 2$; i.e., there are $\Omega(d^2)$ walks of length $2$ from $u$ to this small set $Z$. To be precise, suppose that $|Z| \leq \eps d^2$ and $w_2(u,Z) \geq \gamma d^2$, where $\eps$ and $\gamma$ are constants satisfying $\gamma \geq 4k \eps$, and $d$ is sufficiently large. The number of edges from $N^+(u)$ to $Z$ is precisely $w_2(u,Z) \geq \gamma d^2$, so we can find nonempty subsets $L \subseteq N^+(u)$ and $R \subseteq Z$ such that every vertex in $L$ has out-degree at least $\gamma d/2 \geq 2k$ into the set $R$, and every vertex in $R$ has in-degree at least $(\gamma/\eps)/2 \geq 2k$ from the set $L$.

Now we may find a $(k,2)$-tulip in $D$ as follows (see \Cref{fig:tulip embedding l=2}). We start at $u$, then find a directed path $P$ of length $2$ from $u$ to some vertex $z_0 \in R$. This can be done, for instance, by choosing any vertex $y \in L$ (which is an out-neighbor of $u$ since $L \subseteq N^+(u)$), then choosing an out-neighbor $z_0$ of $y$ in $R$. Next, we  iteratively pick, for each $1 \leq i \leq k-2$, an in-neighbor $y_i$ of $z_{i-1}$ in $L$, then an out-neighbor $z_i$ of $y_i$ in $R$, both distinct from all previously chosen vertices. Finally, we can complete the cycle of length $2k$ by finding a path $P'$ of length $2$ from $u$ to $z_{k-2}$ which is internally vertex-disjoint from all previously chosen vertices. This last step is possible because there is only one internal vertex to choose, which can be chosen from those in-neighbors of $z_{k-2}$ in $L$ (of which we assumed there were at least $2k$).

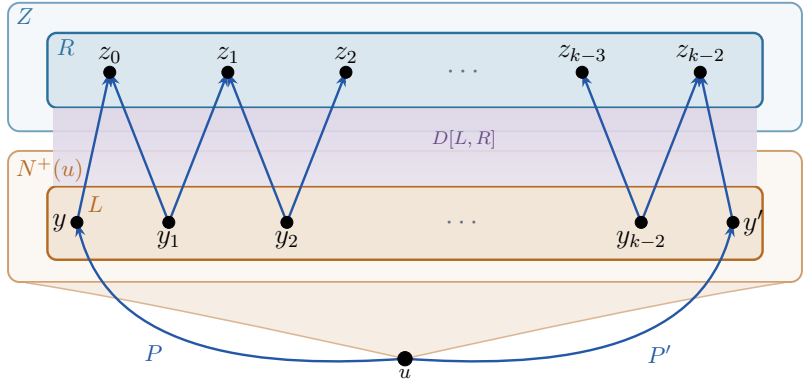
\begin{figure}[h]
\centering
\resizebox{0.65\textwidth}{!}{
\begin{tikzpicture}[
  x=1cm,
  y=1cm,
  >={Stealth[length=2.15mm,width=1.55mm]},
  vertex/.style={circle,draw=black,fill=black,inner sep=0pt,minimum size=5.2pt},
  vertex label/.style={font=\large\bfseries,inner sep=1.2pt,text=black},
  selected edge/.style={->,draw=pathblue,line width=1.05pt},
  fan edge/.style={draw=setorange!65,line width=.68pt,opacity=.72},
  set label/.style={font=\bfseries,inner xsep=1pt,inner ysep=.6pt}
]

% The nested regions are drawn first, so all arrows and labels remain legible.
\begin{scope}[on background layer]
  \path[rounded corners=5pt,fill=setblue!5,draw=setblue!65,line width=.85pt]
    (-6.05,2.12) rectangle (6.05,4.08);
  \path[rounded corners=4pt,fill=setblue!17,draw=setblue,line width=1.05pt]
    (-5.45,2.48) rectangle (5.45,3.62);

  \path[rounded corners=5pt,fill=setorange!5,draw=setorange!65,line width=.85pt]
    (-6.05,-.18) rectangle (6.05,1.82);
  \path[rounded corners=4pt,fill=setorange!17,draw=setorange,line width=1.05pt]
    (-5.45,.16) rectangle (5.45,1.28);
\end{scope}

% Set labels make the two containments visually explicit.
\node[set label,text=setblue,anchor=north west] at (-5.97,4.02) {$Z$};
\node[set label,text=setblue,anchor=north west] at (-5.35,3.55) {$R$};
\node[set label,text=setorange,anchor=north west] at (-5.97,1.74) {$N^+(u)$};
\node[set label,text=setorange,anchor=north west] at (-4.88,1.15) {$L$};

% Coordinates for the vertices in R.
\coordinate (z0)   at (-4.50,3.02);
\coordinate (z1)   at (-2.70,3.02);
\coordinate (z2)   at (-0.90,3.02);
\coordinate (zkm3) at ( 2.70,3.02);
\coordinate (zkm2) at ( 4.50,3.02);

% Coordinates for the vertices in L.
\coordinate (x)    at (-5.00,.73);
\coordinate (y1)   at (-3.60,.73);
\coordinate (y2)   at (-1.80,.73);
\coordinate (ykm2) at ( 3.60,.73);
\coordinate (xp)   at ( 5.00,.73);
\coordinate (u)    at ( 0.00,-1.35);

% A single shaded fan represents all edges from u to its out-neighborhood.
% Its top edge meets the whole N^+(u) region rather than individual vertices.
\begin{scope}[on background layer]
  \path[fill=setorange!12]
    (u)
    .. controls (-1.75,-.96) and (-4.25,-.42) .. (-5.78,-.19)
    -- (5.78,-.19)
    .. controls (4.25,-.42) and (1.75,-.96) .. (u)
    -- cycle;
  \draw[fan edge]
    (u) .. controls (-1.75,-.96) and (-4.25,-.42) .. (-5.78,-.19);
  \draw[fan edge]
    (u) .. controls (1.75,-.96) and (4.25,-.42) .. (5.78,-.19);
\end{scope}

% A shaded curtain represents the ambient high-degree directed graph L -> R.
\begin{scope}[on background layer]
  \path[shade,bottom color=bridgepurple!17,
        top color=bridgepurple!25,opacity=.90]
    (-5.36,1.29) rectangle (5.36,2.47);
\end{scope}
\node[font=\footnotesize\bfseries,text=bridgepurple!85!black] at (.90,2.00)
  {$D[L,R]$};

% First and last length-2 paths.
\draw[selected edge] (u) to[out=184,in=-76]
  node[pos=.58,below left=1pt,font=\bfseries,text=pathblue] {$P$} (x);
\draw[selected edge] (x) -- (z0);
\draw[selected edge] (u) to[out=-4,in=-104]
  node[pos=.58,below right=1pt,font=\bfseries,text=pathblue] {$P'$} (xp);
\draw[selected edge] (xp) -- (zkm2);

% Greedy links: y_i is an in-neighbor of z_{i-1} and an out-neighbor of z_i.
\draw[selected edge] (y1) -- (z0);
\draw[selected edge] (y1) -- (z1);
\draw[selected edge] (y2) -- (z1);
\draw[selected edge] (y2) -- (z2);
\draw[selected edge] (ykm2) -- (zkm3);
\draw[selected edge] (ykm2) -- (zkm2);

% Vertices and labels in R.
\foreach \p/\lab in {
  z0/{z_0},z1/{z_1},z2/{z_2},zkm3/{z_{k-3}},zkm2/{z_{k-2}}
}{
  \node[vertex] at (\p) {};
  \node[vertex label,above=2.5pt] at (\p) {$\lab$};
}
\node[font=\large,text=softgray] at (.92,3.02) {$\cdots$};

% Vertices and labels in L, including the internal vertices of P and P'.
\foreach \p/\lab in {
  y1/{y_1},y2/{y_2},ykm2/{y_{k-2}}
}{
  \node[vertex] at (\p) {};
  \node[vertex label,below=2.5pt] at (\p) {$\lab$};
}
\node[vertex] at (x) {};
\node[vertex label,left=3pt] at (x) {$y$};
\node[vertex] at (xp) {};
\node[vertex label,right=3pt] at (xp) {$y'$};
\node[font=\large,text=softgray] at (.90,.73) {$\cdots$};

% Root.
\node[vertex,minimum size=6.2pt] at (u) {};
\node[vertex label,below=3pt,font=\small\bfseries] at (u) {$u$};

% \node[anchor=north,align=center,text width=12cm,font=\small] at (0,-1.98) {A greedy embedding of a $(k,2)$-tulip rooted at $u$, through
%   $L\subseteq N^+(u)$ and $R\subseteq Z$.};

\end{tikzpicture}
}

\caption{A greedy embedding of a $(k,2)$-tulip rooted at $u$ in $D$ through the special sets $L \subseteq N^+(u)$ and $R \subseteq Z$ in the case $\ell=2$.
}\label{fig:tulip embedding l=2}
\end{figure}

We can carry out a similar strategy if $\ell > 2$ (see \Cref{fig:tulip embedding} below), except that we need to guarantee additional structure in order to perform the first and last steps, where we wish to find paths of length $\ell$ from $u$ to $R$ which avoid vertex collisions. In general, what we want to do is find a pair of subsets $L, R \subseteq V(D)$ such that every vertex in $L$ has large out-degree into $R$, every vertex in $R$ has large in-degree from $L$, and every vertex in $R$ can be \emph{robustly} reached from $u$ by walks of length $\ell$, so that we can close up the cycle in the last step without reusing vertices. The following definition makes precise the type of robustness that we will require, and the main lemma which finds these structures is proved at the beginning of the next subsection as \Cref{lem:robust bipartite graph for digraph}.

\begin{definition}\label{def:robust graph}
    Let $D$ be a fixed digraph, let $\ell$ be a positive integer, and let $\delta > 0$. A pair of vertices $(u,z)$ in $D$ is called \emph{$(\delta, \ell)$-robust} with respect to $D$ if the following hold:
    \stepcounter{propcounter}
\begin{enumerate}[label = {{{\normalfont\textbf{\Alph{propcounter}\arabic{enumi}}}}}]
    \item \label{prop:robust vertex} For every vertex $y$ and every $1 \leq t \leq \ell-1$, less than a $\delta$-proportion of walks of length $\ell$ from $u$ to $z$ have $y$ as the $t$-th vertex; that is,
    \[w_t(u,y)w_{\ell-t}(y,z) < \delta w_\ell(u,z).\]
    \item \label{prop:robust paths} Less than a $\delta$-proportion of walks of length $\ell$ from $u$ to $z$ are degenerate; that is, there are more than $(1-\delta)w_\ell(u,z)$ directed paths of length $\ell$ from $u$ to $z$.
\end{enumerate}
\end{definition}

\subsection{The \texorpdfstring{Bondy--Simonovits}{Bondy-Simonovits} Theorem}\label{sec:bondy proof}
We now illustrate our method in full detail by proving the Bondy--Simonovits Theorem (\Cref{thm:bondy-simonovits}) and its supersaturation version (\Cref{thm:simonovits cycle supersaturation}).

We first state and prove our main lemma which finds the special integer $\ell$, this special vertex $u$, and these special sets $L$ and $R$ as described in the previous subsection.

\begin{lemma}\label{lem:robust bipartite graph for digraph}
    For every $\delta > 0$ and integers $b \geq 1$ and $k \geq 2$, there exist positive constants $C = C(k, b, \delta)$ and $\beta = \beta(k, b, \delta)$ such that the following holds for every $n$-vertex, $d$-out-regular directed graph $D$ with $d \geq Cn^{1/k}$. There exist a vertex $u$, an integer $\ell$ with $2 \leq \ell \leq k$ and two subsets $L, R \subseteq V(D)$ satisfying:
    \stepcounter{propcounter}
    \begin{enumerate}[label = {{{\normalfont\textbf{\Alph{propcounter}\arabic{enumi}}}}}]
        \item\label{prop:right robust graph} For every element $z \in R$, the pair $(u,z)$ is $(\delta, \ell)$-robust with respect to $D$.
        \item\label{prop:weight count graph} $w_{\ell}(u,R) \geq \beta d^{\ell}$.
        \item\label{prop:right degree graph} Every vertex $z \in R$ has at least $b$ in-neighbors in $L$.
        \item\label{prop:left degree graph} Every vertex $x \in L$ has at least $\beta d$ out-neighbors  in $R$.
    \end{enumerate}
\end{lemma}
\begin{proof}
We may assume $\delta \leq 1$. Define parameters $1 = \gamma_k > \cdots > \gamma_1 > \eps_1 > \cdots > \eps_k$ by $\gamma_t := (\delta/16)^{k-t}$ and $\eps_t = (\gamma_1/(6b))^{t}$ for all $1 \leq t \leq k$. Define $C := \eps_k^{-1/k}$.  Let $D$ be an $n$-vertex, $d$-out-regular digraph with $d \geq Cn^{1/k}$. By increasing the value of~$C$ if necessary, we may assume $n$ is sufficiently large. Let $\ell \in [k]$ be minimum with the following property: 
    \[\label{eq:star graph}\text{There exist $u \in V(D)$ and $Z \subseteq V(D)$ of size $|Z| \leq \eps_\ell d^{\ell} $ with $\ell$-weight $w_{\ell}(u,Z) \geq \gamma_\ell d^\ell$.}\tag{$*$}\]
Such an $\ell$ exists since $\eps_k d^k \geq \eps_k C^k n = |V(D)|$ and $w_k(u,V(D)) = d^k = \gamma_k d^k$ for any $u \in V(D)$.
On the other hand, we must have $\ell \geq 2$
because $\eps_1 < \gamma_1$, and $w_1(u,Z) \leq |Z|$ for any $u \in V(D)$ and $Z \subseteq V(D)$.
Let $u \in V(D)$ and $Z \subseteq V(D)$ be the vertex and vertex set guaranteed by~\eqref{eq:star graph} for this choice of $\ell$.

    Let $\hat R$ be the set of vertices $z \in Z$ such that $(u,z)$ is $(\delta, \ell)$-robust with respect to $D$. Our first goal is to show that at least half of the assumed $\ell$-weight of $Z$ with respect to $u$ comes from $\hat R$.
    \begin{claim}\label{cla:robust weight digraph}
        $w_{\ell}(u, \hat R) \geq (\gamma_\ell/2)d^{\ell}$.
    \end{claim}
    \begin{subproof}
        Let $B_0$ be the set of vertices $z \in Z$ such that there are at least $\delta w_{\ell}(u,z)$ degenerate walks from $u$ to $z$ of length $\ell$. Since the total number of degenerate walks of length $\ell$ in $D$ beginning at $u$ is at most $\ell^2 d^{\ell-1}$, we have that
        \begin{equation}\label{eq:degenerate weight}
            w_\ell(u,B_0) \leq \delta^{-1}\ell^2d^{\ell-1} \leq \frac{\gamma_\ell}{4}d^\ell.
        \end{equation}
        Here we used that $d \geq Cn^{1/k}$ and $n$ is sufficiently large.
        
        Note that for every vertex $z \in Z \setminus (\hat R \cup B_0)$, by \Cref{def:robust graph}, there exist an integer $t = t_z \in [\ell-1]$ and a vertex $y = y_z$ such that $w_t(u,y)w_{\ell-t}(y,z) \geq \delta w_\ell(u,z)$. If more than one such pair $(t,y)$ exists, then we may designate one such pair as $(t_z, y_z)$ arbitrarily.
    Now, for each pair $(t,y) \in [\ell-1] \times V(D)$, define
    \[B_{t,y} := \{z \in Z \setminus (\hat R \cup B_0) : (t_z,y_z) = (t,y)\}.\]
    Further define, for each $t \in [\ell-1]$, the sets
    \begin{align*}
        Y_{t, \mathrm{small}} &:= \left\{y \in V(D) : |B_{t,y}| \leq \eps_{\ell-t}d^{\ell-t}\right\}, \\
        Y_{t, \mathrm{large}} &:= \left\{y \in V(D) : |B_{t,y}| > \eps_{\ell-t}d^{\ell-t}\right\}.
    \end{align*}
    Then we have
    \begin{align}
        w_\ell(u,Z \setminus (\hat R \cup B_0)) &= \sum_{t=1}^{\ell-1} \sum_{y \in V(D)}\sum_{z \in B_{t,y}}w_\ell(u,z) \nonumber \\
        &\leq \sum_{t=1}^{\ell-1} \sum_{y \in V(D)}\sum_{z \in B_{t,y}}\delta^{-1} w_t(u,y)w_{\ell-t}(y,z) \nonumber \\
        &= \sum_{t=1}^{\ell-1} \delta^{-1}\left(\sum_{y \in Y_{t,\mathrm{small}}}w_t(u,y)w_{\ell-t}(y,B_{t,y}) + \sum_{y \in Y_{t,\mathrm{large}}}w_t(u,y)w_{\ell-t}(y,B_{t,y})\right). \label{eq:organizing weight graph}
    \end{align}
    Observe that $w_{\ell-t}(y,B_{t,y}) \leq d^{\ell-t}$ for every $t \in [\ell-1]$ and $y \in V(D)$, and if $y \in Y_{t,\mathrm{small}}$, then it follows from our choice of $\ell$ as the minimum satisfying~\eqref{eq:star graph} that $w_{\ell-t}(y,B_{t,y}) \leq \gamma_{\ell-t}d^{\ell-t}$. Thus, by~\eqref{eq:organizing weight graph},
    \begin{align}
        w_\ell(u,Z \setminus (\hat R \cup B_0)) 
        &\leq \sum_{t=1}^{\ell-1} \delta^{-1}\left(w_t(u,Y_{t,\mathrm{small}})(\gamma_{\ell-t}d^{\ell-t}) + w_t(u,Y_{t,\mathrm{large}})d^{\ell-t}\right).\label{eq:weight step 2 graph}
    \end{align}
    For each $t \in [\ell-1]$, because the sets $B_{t,y}$ for $y \in Y_{t,\mathrm{large}}$ are disjoint subsets of $Z$ of size at least $\eps_{\ell-t}d^{\ell-t}$, we have that 
    \[|Y_{t, \mathrm{large}}| \leq \frac{|Z|}{\eps_{\ell-t} d^{\ell-t}} \leq \frac{\eps_\ell d^\ell}{\eps_{\ell-t}d^{\ell-t}} = \eps_t d^t\]
    by~\eqref{eq:star graph}, so by the minimality of $\ell$, we have $w_t(u,Y_{t,\mathrm{large}}) \leq \gamma_td^t$.  We also trivially have $w_t(u,Y_{t,\mathrm{small}}) \leq d^t$, so by~\eqref{eq:weight step 2 graph}, we have
    \begin{align*}
        w_\ell(u,Z \setminus (\hat R \cup B_0)) &\leq \sum_{t=1}^{\ell-1} \delta^{-1}\left(\gamma_{\ell-t}d^\ell + \gamma_td^{\ell}\right) \leq \frac{\gamma_\ell}{4} d^{\ell}.
    \end{align*}
    Combining this with~\eqref{eq:degenerate weight} and~\eqref{eq:star graph} now gives
    \[w_{\ell}(u,\hat R) = w_\ell(u,Z) - w_\ell(u,B_0) - w_{\ell}(u,Z \setminus (\hat R \cup B_0)) \geq \frac{\gamma_\ell}{2}d^{\ell},\]
    as claimed.
    \end{subproof}
    Now let $\hat L$ be the set of vertices $y \in V(D)$ with positive $(\ell-1)$-weight, and note that $|\hat L| \leq d^{\ell-1}$. Also recall that $|\hat R| \leq |Z| \leq \eps_\ell d^{\ell}$ by~\eqref{eq:star graph}. We construct an auxiliary bipartite graph $\hat H$ on $(\hat L, \hat R)$ by joining $y \in \hat L$ to $z \in \hat R$ whenever $yz \in E(D)$. Here, we treat $\hat L$ and $\hat R$ as disjoint sets in $\hat H$, even though they may not be disjoint in $D$. We want to give a lower bound on the number of edges of $\hat H$. To that end, let $Y := \{y \in \hat L : w_{\ell-1}(u,y) \geq 2\gamma_{\ell-1}/\eps_{\ell-1}\}$, and note that $Y$ cannot be larger than $\eps_{\ell-1}d^{\ell-1}$ by our choice of~$\ell$ as the smallest satisfying~\eqref{eq:star graph}. Consequently, $w_{\ell-1}(u,Y) \leq \gamma_{\ell-1}d^{\ell-1}$, also by our choice of $\ell$. Therefore, since every vertex in $\hat L$ has degree at most $d$ in $\hat H$, we have
    \begin{align*}
        w_{\ell}(u,\hat R) &= \sum_{y \in \hat L} w_{\ell-1}(u,y)\deg_{\hat H}(y) \\
        &\leq w_{\ell-1}(u,Y)d + \sum_{y \in \hat L \setminus Y} \frac{2\gamma_{\ell-1}}{\eps_{\ell-1}}\deg_{\hat H}(y) \\
        &\leq \gamma_{\ell-1}d^{\ell} + \frac{2\gamma_{\ell-1}}{\eps_{\ell-1}}e(\hat H).
    \end{align*}
    
    Rearranging and applying \Cref{cla:robust weight digraph}, we have
    \begin{equation}\label{eq:edges in first graph digraph}
        e(\hat H) \geq \frac{\eps_{\ell-1}w_{\ell}(u,\hat R)}{2\gamma_{\ell-1}} - \frac{\eps_{\ell-1}}{2}d^\ell \geq \left(\frac{\eps_{\ell-1}\gamma_{\ell}}{4\gamma_{\ell-1}} - \frac{\eps_{\ell-1}}{2}\right)d^\ell \geq 3b\eps_{\ell}d^\ell.
    \end{equation}
    Here we used that $\gamma_{\ell-1} \leq \gamma_{\ell}/4$ and $\eps_{\ell} \leq \eps_{\ell-1}/(6b)$.

    Next, we iteratively delete from $\hat H$ any vertex in $\hat L$ with degree less than $b\eps_{\ell}d$ and any vertex in $\hat R$ with degree less than $b$, as long as possible. Let $H$ be the resulting graph, and observe that the number of edges deleted is less than 
    \[b \eps_\ell d|\hat L| + b|\hat R| \leq  2b\eps_{\ell}d^{\ell},\]
    so $e(H) \geq b \eps_{\ell}d^{\ell}$ by~\eqref{eq:edges in first graph digraph}. In particular, $H$ is nonempty, every vertex in $L := \hat L \cap V(H)$ has degree at least~$b\eps_{\ell}d$, and every vertex in $R := \hat R \cap V(H)$ has degree at least $b$. Moreover, because every vertex in $L$ has positive $(\ell-1)$-weight, it follows that
    \[w_{\ell}(u,R) \geq \sum_{y \in L}w_{\ell-1}(u,y)\deg_H(y) \geq e(H) \geq b\eps_{\ell} d^{\ell}.\]
    Lastly, for every element $z \in R \subseteq \hat R$, the pair $(u,z)$ is $(\delta, \ell)$-robust by the definition of $\hat R$.
    Thus~\ref{prop:right robust graph}--\ref{prop:left degree graph} are satisfied with~$\beta := b\eps_k$.
\end{proof}

We now use the previous lemma to find not only one, but $\Omega(d^k)$ oriented cycles of length $2k$ in an $n$-vertex, $d$-out-regular digraph with $d \geq Cn^{1/k}$. Specifically, we will find $\Omega(d^k)$ many $(k,\ell)$-tulips rooted at a common vertex $u$, for some $2 \leq \ell \leq k$. Since each such tulip uses two out-edges from $u$, this will allow us to delete a small set of edges and iterate in order to increase this count even further.
\begin{lemma}\label{lem:many oriented cycles}
    For every $k \geq 2$, there exist constants $C, c > 0$ such that every $n$-vertex, $d$-out-regular directed graph $D$ with $d \geq Cn^{1/k}$ contains a vertex $u$ and at least $cd^k$ many $(k,\ell)$-tulips rooted at $u$, for some $2 \leq \ell \leq k$.
\end{lemma}
\begin{proof}
    Choose $\delta := 1/(4k^2)$ and $b := 2k$. Let $C = C(k,b,\delta)$ and $\beta = \beta(k,b,\delta)$ be as in \Cref{lem:robust bipartite graph for digraph}. We may assume that $n$ is sufficiently large. By \Cref{lem:robust bipartite graph for digraph}, there exist a vertex $u \in V(D)$, an integer $2 \leq \ell \leq k$, and two subsets $L,R \subseteq V(D)$ satisfying \ref{prop:right robust graph}--\ref{prop:left degree graph}.

    We attempt to construct many $(k,\ell)$-tulips rooted at $u$ as follows. Initialize $U := \{u\}$.
    \begin{enumerate}
        \item Choose a directed path $P = (u, v_1, \ldots, v_{\ell-1}, z_0)$ of length $\ell$ from $u$ to some vertex $z_0 \in R$. Update $U \leftarrow U \cup \{v_1, \ldots, v_{\ell-1}, z_0\}$.
        \item For $1 \leq i \leq k-\ell$, in turn, choose an in-neighbor $y_i$ of $z_{i-1}$ in $L \setminus U$, then choose an out-neighbor $z_i$ of $y_i$ in $R \setminus U$. Update $U \leftarrow U \cup \{y_i, z_i\}$.
        \item Choose a directed path $P' = (u, v_1', \ldots, v_{\ell-1}', z_{k-\ell})$ of length $\ell$ from $u$ to $z_{k-\ell}$ such that the internal vertices $v_1', \ldots, v_{\ell-1}'$ are not in $U$. Update $U \leftarrow U \cup \{v_1', \ldots, v_{\ell-1}'\}$.
    \end{enumerate}
\begin{figure}[h]
\centering
% Scale the drawing without shrinking the labels.
\pgfmathsetlengthmacro{\tulipunit}{0.35*\textwidth/8.3}
\begin{tikzpicture}[yscale=0.9,
  x=\tulipunit,
  y=\tulipunit,
  >={Stealth[length=2.15mm,width=1.55mm]},
  vertex/.style={circle,draw=black,fill=black,inner sep=0pt,
                 minimum size=5.2pt,outer sep=.4pt},
  vertex label/.style={font=\bfseries,inner sep=1.2pt,text=black},
  selected edge/.style={->,draw=pathblue,line width=1.05pt},
  path label/.style={font=\bfseries,inner sep=1.2pt,text=pathblue},
  set label/.style={font=\bfseries,inner xsep=1pt,inner ysep=.6pt}
]

% Shaded vertex classes.
\begin{scope}[on background layer]
  \path[rounded corners=4pt,fill=setblue!17,
        draw=setblue,line width=1.05pt]
    (-4.15,6.30) rectangle (4.15,7.55);
  \path[rounded corners=4pt,fill=setorange!17,
        draw=setorange,line width=1.05pt]
    (-4.15,4.97) rectangle (4.15,6.04);
\end{scope}
\node[set label,text=setblue,anchor=east]
  at (-4.32,6.85) {$R$};
\node[set label,text=setorange,anchor=east]
  at (-4.32,5.59) {$L$};

% Root and the displayed internal vertices of P and P'.
\node[vertex,minimum size=6.2pt] (u) at (0,-.70) {};
\foreach \i/\xx/\yy in {
  1/1.80/.10,
  2/3.00/1.45,
  3/3.55/3.00,
  4/3.65/4.55
}{
  \node[vertex] (l\i) at (-\xx,\yy) {};
  \node[vertex] (r\i) at ( \xx,\yy) {};
}

% Vertices in R.
\foreach \p/\xx in {
  z0/-3.45,z1/-2.20,z2/-.95,z3/.30,
  zprev/2.10,zlast/3.45
}{
  \node[vertex] (\p) at (\xx,6.85) {};
}

% Vertices in L.
\foreach \p/\xx in {
  y1/-2.825,y2/-1.575,y3/-.325,ylast/2.775
}{
  \node[vertex] (\p) at (\xx,5.59) {};
}

% Rounded sides of the tulip.
\draw[selected edge] (u)  to[out=170,in=-35] (l1);
\draw[selected edge] (l1) to[out=145,in=-70] (l2);
\draw[selected edge] (l3) to[out=100,in=-90] (l4);
\draw[selected edge] (l4) to[out=90,in=-98] (z0);

\draw[selected edge] (u)  to[out=10,in=215] (r1);
\draw[selected edge] (r1) to[out=35,in=250] (r2);
\draw[selected edge] (r3) to[out=80,in=270] (r4);
\draw[selected edge] (r4) to[out=90,in=278] (zlast);

% Omitted internal vertices.
\foreach \side in {l,r}{
  \path (\side2) -- (\side3)
    node[midway,sloped,font=\large,text=softgray] {$\cdots$};
}

% Greedy links.
\foreach \i/\j in {1/0,2/1,3/2}{
  \draw[selected edge] (y\i) -- (z\j);
  \draw[selected edge] (y\i) -- (z\i);
}
\draw[selected edge] (ylast) -- (zprev);
\draw[selected edge] (ylast) -- (zlast);

% Vertex labels.
\foreach \p/\lab in {
  z0/{z_0},z1/{z_1},z2/{z_2},z3/{z_3},
  zprev/{z_{k-\ell-1}},zlast/{z_{k-\ell}}
}{
  \node[vertex label,above=2.5pt] at (\p) {$\lab$};
}
\foreach \p/\lab in {
  y1/{y_1},y2/{y_2},y3/{y_3},ylast/{y_{k-\ell}}
}{
  \node[vertex label,below=2.5pt] at (\p) {$\lab$};
}
\node[font=\large,text=softgray] at (1.18,6.85) {$\cdots$};
\node[font=\large,text=softgray] at (1.18,5.59) {$\cdots$};

% Path labels and root label.
\node[path label,left=5pt] at (l3) {$P$};
\node[path label,right=5pt] at (r3) {$P'$};
\node[vertex label,below=3pt] at (u) {$u$};

\end{tikzpicture}
\caption{A $(k,\ell)$-tulip embedded in $D$ using Steps~1--3.
Step~1 embeds the directed path $P$ from $u$ to $z_0$.
Step~2 embeds the vertices $y_1,z_1,\ldots,y_{k-\ell},z_{k-\ell}$
one at a time.
Step~3 embeds the directed path $P'$ from $u$ to $z_{k-\ell}$.}
\label{fig:tulip embedding}
\end{figure}
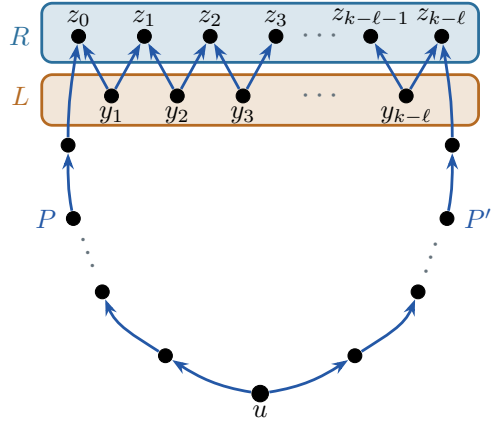

    Clearly, if this procedure succeeds, then we have found a $(k,\ell)$-tulip in $D$ rooted at $u$.
    Now let us count the number of choices we have at each of these steps.
    
    \textbf{Step 1.} By~\ref{prop:right robust graph}, every vertex in $R$ is $(\delta,\ell)$-robust with respect to $u$, so by~\ref{prop:robust paths}, there are at least $(1-\delta)w_{\ell}(u,R)$ directed paths of length $\ell$ from $u$ to $R$. By~\ref{prop:weight count graph}, this accounts for at least 
    \begin{equation}\label{eq:step 1 count graph}
        (1-\delta)\beta d^\ell \geq \frac{\beta}{2} d^{\ell}
    \end{equation}
    choices for $P$ in Step~$1$. 

    \textbf{Step 2.} Fix $1 \leq i \leq k-\ell$. Suppose we have already chosen $y_1, z_1, \ldots, y_{i-1}, z_{i-1}$, and we wish to choose $y_i$ and $z_i$. Since $z_{i-1} \in R$, we have by~\ref{prop:right degree graph} that there are at least $b = 2k$ choices of $y_i \in L$ such that $y_iz_{i-1} \in E(D)$. Of these, at most $|U| = \ell+1 + 2(i-1) < 2k$ are in $U$, leaving us at least one good choice for $y_i$. Similarly, by~\ref{prop:left degree graph}, there are at least $\beta d - 2k$ good choices for $z_i$. Altogether, this gives at least
    \begin{equation}\label{eq:step 2 count graph}
        \left(\beta d - 2k\right)^{k-\ell} \geq \left(\frac{\beta}{2}\right)^{k-\ell}d^{k-\ell}
    \end{equation}
    choices for $y_1, z_1, \ldots, y_{k-\ell}, z_{k-\ell}$ in Step 2.

    \textbf{Step 3.} Since $z_{k-\ell} \in R$, we have by~\ref{prop:right robust graph} that $z_{k-\ell}$ is $(\delta, \ell)$-robust. Therefore, by~\ref{prop:robust vertex}, less than $(\ell-1)|U|\delta w_\ell(u,z_{k-\ell}) \leq w_\ell(u,z_{k-\ell})/2$ walks of length $\ell$ from $u$ to $z_{k-\ell}$ use an internal vertex from $U$, and by~\ref{prop:robust paths}, less than $\delta w_\ell(u,z_{k-\ell}) \leq w_\ell(u,z_{k-\ell})/2$ walks of length $\ell$ from $u$ to $z_{k-\ell}$ are degenerate. This leaves a positive number of choices for the directed path $P'$ in Step 3.

    Now putting it all together, we have by~\eqref{eq:step 1 count graph} and~\eqref{eq:step 2 count graph} that there are at least 
    \[\frac{\beta}{2} d^{\ell} \cdot \left(\frac{\beta}{2}\right)^{k-\ell}d^{k-\ell} = \left(\frac{\beta}{2}\right)^{k-\ell+1}d^k\]
    ways to choose a sequence $(P, y_1, z_1, \ldots, y_{k-\ell}, z_{k-\ell}, P')$ as in Steps 1--3. No more than two of these give rise to the same $(k,\ell)$-tulip, so there are at least $(\beta/2)^{k-\ell+1}d^k/2$ many $(k,\ell)$-tulips in $D$ rooted at $u$.
\end{proof}

Now we again consider undirected graphs. The previous lemma has the following consequence, which already recovers the Bondy--Simonovits Theorem (\Cref{thm:bondy-simonovits}).
\begin{lemma}\label{lem:many cycles at a vertex}
    For every $k \geq 2$, there exist constants $C, c > 0$ such that every $n$-vertex graph $G$ with at least $Cn^{1+1/k}$ edges has a set $E$ of at most $2Cn^{1/k}$ edges (sharing a common vertex) such that at least $cn$ cycles of length $2k$ in $G$ contain an edge of $E$.
\end{lemma}
\begin{proof}
    Let $C,c' > 0$ be the constants given by \Cref{lem:many oriented cycles} with $c$ replaced by $c'$. Let $G$ be an $n$-vertex graph with at least $Cn^{1+1/k}$ edges. Then $G$ has a nonempty subgraph $G'$ on $n' \leq n$ vertices with minimum degree at least $Cn^{1/k}$. Let~$D'$ be the digraph obtained from~$G'$ by replacing each edge with two antiparallel directed edges on the same pair of vertices. Then~$D'$ has minimum out-degree at least $Cn^{1/k}$ and contains a $d$-out-regular subgraph~$D$, where $d := \lceil Cn^{1/k} \rceil \geq C(n')^{1/k}$. By \Cref{lem:many oriented cycles}, $D$ contains a vertex $u$ and at least $c'd^k \geq c'C^kn$ many $(k,\ell)$-tulips rooted at $u$. Let $E'$ be the set of edges directed away from $u$ in~$D$, and let $E$ be the corresponding set of edges in $G$. The number of such edges is at most $d \leq 2Cn^{1/k}$ since $D$ is $d$-out-regular. Now every $(k,\ell)$-tulip rooted at $u$ must use some edge of $E'$ (two edges, in fact), so every $(k,\ell)$-tulip rooted at $u$ corresponds to a (unique) cycle of length $2k$ in $G$ containing an edge of $E$. Thus the statement of the lemma holds for $c := c'C^k$.
\end{proof}
The rest of this section is devoted to proving the Erd\H{o}s-Simonovits supersaturation result for even cycles, where we want to find $\Omega((e(G)/n)^{2k})$ cycles of length $2k$ in any $n$-vertex graph $G$ with at least $Cn^{1+1/k}$ edges. First, we can iteratively apply the previous lemma in order to increase the number of cycles of length $2k$ from $\Omega(n)$ to $\Omega(n^2)$, which is the correct count for \Cref{thm:simonovits cycle supersaturation} when the number of edges is $\Theta(n^{1+1/k})$.
\begin{lemma}\label{lem:cycles at the threshold}
    For every $k \geq 2$, there exist constants $C, c > 0$ such that every $n$-vertex graph $G$ with at least $Cn^{1+1/k}$ edges has at least $cn^2$ cycles of length $2k$.
\end{lemma}
\begin{proof}
    Let $C',c > 0$ be the constants given by \Cref{lem:many cycles at a vertex} with $C$ replaced by $C'$. Let $G$ be an $n$-vertex graph with at least $3C' n^{1+1/k}$ edges.
    
    Define $G_0 := G$. For each $i \geq 1$ such that $G_{i-1}$ has at least $C'n^{1+1/k}$ edges, define a spanning subgraph $G_i \subseteq G_{i-1}$ as follows. By \Cref{lem:many cycles at a vertex}, there exists a set $E_{i-1}$ of at most $2C'n^{1/k}$ edges in $G_{i-1}$ and a collection $\mathcal C_{i-1}$ of at least $cn$ cycles of length $2k$ in $G_{i-1}$ such that every cycle in $\mathcal C_{i-1}$ contains some edge of~$E_{i-1}$. Now take $G_i := G_{i-1} \setminus E_{i-1}$.

    Let $i_0 \geq 0$ be the greatest index for which $G_{i_0}$ is defined. Since $G_{i_0+1}$ was not defined, we must have that $e(G_{i_0}) < C'n^{1+1/k}$. Therefore,
    \[2C'n^{1+1/k}\leq e(G_0) - e(G_{i_0}) \leq \sum_{i=0}^{i_0-1} |E_i| \leq i_0(2C'n^{1/k}),\]
    and so $i_0 \geq n$.

    Note that by design, the sets $\mathcal C_0, \ldots, \mathcal C_{i_0-1}$ are disjoint collections of cycles of length $2k$ in $G$, so there are at least
    \[\sum_{i=0}^{i_0-1} |\mathcal C_i| \geq cn^2\]
    cycles of length $2k$ in $G$ in total, as desired.
\end{proof}
Now the standard sampling argument discussed at the beginning of \Cref{sec:prelim} lifts this supersaturation result for graphs with $\Theta(n^{1+1/k})$ edges to a supersaturation result for graphs with $\Omega(n^{1+1/k})$ edges, giving \Cref{thm:simonovits cycle supersaturation}.

\begin{proof}[Proof of \Cref{thm:simonovits cycle supersaturation}]

Let $C, c' > 0$ be the constants given by \Cref{lem:cycles at the threshold} with $c'$ in place of $c$. Fix an $n$-vertex graph $G$ with at least $Cn^{1+1/k}$ edges. Let $\mathcal H$ be the $2k$-uniform hypergraph on vertex set $E(G)$ whose hyperedges are the edge sets of cycles of length $2k$ in $G$. If $G' \subseteq G$ is a spanning subgraph with exactly $m := \lceil C n^{1+1/k} \rceil$ edges, then by \Cref{lem:cycles at the threshold}, $G'$ contains at least
    \[c'n^2 = \frac{c'n^2m^{2k}}{m^{2k}} \geq \frac{c'n^2m^{2k}}{(2Cn^{1+1/k})^{2k}} = \frac{c'm^{2k}}{(2C)^{2k}n^{2k}}\]
    cycles of length $2k$, and $E(G')$ contains at least the same number of hyperedges of $\mathcal H$. Thus, by \Cref{prop:supersaturation lifting hypergraph}, the number of cycles of length $2k$ in $G$ is at least
    \[e(\mathcal H) \geq \frac{c'e(G)^{2k}}{(2C)^{2k}n^{2k}},\]
    so the theorem holds with $c := c'/(2C)^{2k}$.
\end{proof}

\section{Inseparable solutions}\label{sec:inseparable}
In the next two sections, we will prove Theorems~\ref{thm:additively generic integers} and~\ref{thm:inseparable and additively generic finite group}, which involve finding inseparable, and in some cases additively generic, solutions to~\eqref{eq:equal k-sums} in integer and finite Abelian group settings. These theorems will be deduced from the following more general statements which hold for both finite and infinite Abelian groups. The first involves inseparable solutions in arbitrary Abelian groups, and it will be proved in this section, while the second involves additively generic solutions in Abelian groups with no $2$-torsion, and its proof will be given in \Cref{sec:additively generic}.

\begin{theorem}\label{thm:inseparable small growth}
    For every integer $k \geq 2$, there exist constants $\eps, c > 0$ such that for every Abelian group $\Gamma$ and finite subset $A \subseteq \Gamma$ with $|A^{+k}| \leq \eps |A|^k$, there are at least $c|A|^k$ inseparable solutions $(x_1, \ldots, x_{2k}) \in A^{2k}$ to the equation
    \[x_1 + \cdots + x_k = x_{k+1} + \cdots + x_{2k} \tag{\ref{eq:equal k-sums}}.\]
\end{theorem}

\begin{theorem}\label{thm:truly additively generic small growth}
    For every integer $k \geq 2$, there exist constants $\eps, c > 0$ such that for every Abelian group~$\Gamma$ with no $2$-torsion and finite subset $A \subseteq \Gamma$ with $|A^{+k}| \leq \eps |A|^k$, there are at least $c|A|^k$ additively generic solutions $(x_1, \ldots, x_{2k})$ to the equation
    \[x_1 + \cdots + x_k = x_{k+1} + \cdots + x_{2k}. \tag{\ref{eq:equal k-sums}}\]
\end{theorem}

Assuming these two theorems, we now use \Cref{prop:supersaturation lifting hypergraph} to deduce Theorems~\ref{thm:additively generic integers} and~\ref{thm:inseparable and additively generic finite group}. We begin in the integer setting.

\begin{proof}[Proof of \Cref{thm:additively generic integers}]
    Let $\eps, c' > 0$ be the constants given by \Cref{thm:truly additively generic small growth} with $c'$ in place of $c$, and define $C := (\eps/k)^{-1/k}$. Fix a positive integer $N$, and let $A \subseteq [N]$ with size at least $CN^{1/k}$. Let $\mathcal H$ be the $2k$-uniform hypergraph on vertex set $A$ whose edges are the subsets $\{x_1, \ldots, x_{2k}\} \subseteq A$ such that $(x_1, \ldots, x_{2k})$ is an additively generic solution to~\eqref{eq:equal k-sums}. If $A' \subseteq A$ has size exactly $m := \lceil C N^{1/k} \rceil$, then 
    \[|(A')^{+k}| \leq kN  = \eps C^k N  \leq \eps |A'|^k,\]
    so it follows from \Cref{thm:truly additively generic small growth} that $A' \subseteq V(\mathcal H)$ contains at least $c'm^k/(2k)!\geq c' m^{2k}/((2k)!(2C)^kN)$ hyperedges. Thus, by \Cref{prop:supersaturation lifting hypergraph}, the number of additively generic solutions to~\eqref{eq:equal k-sums} in $A$ is at least $e(\mathcal H) \geq c'|A|^{2k}/((2k)! (2C)^kN),$
    so the theorem holds with $c := c'/((2k)!(2C)^k)$.
\end{proof}

Next we consider the finite group setting.

\begin{proof}[Proof of \Cref{thm:inseparable and additively generic finite group}]
 Let $\eps, c' > 0$ be the smaller of the constants given by \Cref{thm:inseparable small growth} and \Cref{thm:truly additively generic small growth} with~$c'$ in place of $c$, and define $C := \eps^{-1/k}$. Fix a finite Abelian group $\Gamma$ of order $N$, and let $A \subseteq \Gamma$ with size at least $CN^{1/k}$. Observe that if $N$ is odd, then $\Gamma$ has no $2$-torsion. Let $\mathcal H$ be the $2k$-uniform hypergraph on vertex set $A$ whose edges are the subsets $\{x_1, \ldots, x_{2k}\} \subseteq A$ such that $(x_1, \ldots, x_{2k})$ is an inseparable solution to~\eqref{eq:equal k-sums} if $N$ is even, or an additively generic solution to~\eqref{eq:equal k-sums} if $N$ is odd. If $A' \subseteq A$ has size exactly $m := \lceil C N^{1/k} \rceil$, then 
    \[|(A')^{+k}| \leq N  = \eps C^k N  \leq \eps |A'|^k,\]
    so it follows from Theorems~\ref{thm:inseparable small growth} and~\ref{thm:truly additively generic small growth} that $A' \subseteq V(\mathcal H)$ contains at least $c'm^k/(2k)! \geq c' m^{2k}/((2k)!(2C)^kN)$ hyperedges.
    Thus, by \Cref{prop:supersaturation lifting hypergraph}, we have $e(\mathcal H) \geq c|A|^{2k}/N$, where $c:= c'/((2k)!(2C)^k)$. This means that $A$ contains at least $c|A|^{2k}/N$ solutions to~\eqref{eq:equal k-sums} which are inseparable if $N$ is even and additively generic if $N$ is odd.
\end{proof}

The remainder of this section is devoted to proving \Cref{thm:inseparable small growth}. The proof technique we use is essentially equivalent to applying the argument given in the previous section to the edge-colored Cayley digraph $D = \mathrm{Cay}(\Gamma,A)$ on $\Gamma$ generated by the set $A$, which is defined to be the digraph on vertex set $\Gamma$ and color set $A$ with an edge from $g$ to $g + a$ of color $a$ for every $g \in \Gamma$ and $a \in A$. Here, the inseparable solutions to~\eqref{eq:equal k-sums} in $A$ that we look for correspond to $(k,\ell)$-tulips in $D$ for some $2 \leq \ell \leq k$, with some additional structure. The number of walks of length $\ell$ from a vertex $u$ to a vertex $z$ in $D$ is the number of representations of the element $z-u$ as a sum of $\ell$ elements of $A$, and this is how we will define our weight function.
\begin{definition}\label{def:weight inseparable}
    Let $\Gamma$ be an Abelian group, let $A$ be a fixed finite subset of $\Gamma$, and let $\ell$ be a positive integer. The \emph{$\ell$-weight} of an element $z \in \Gamma$, denoted $w_\ell(z)$, is the number of ordered $\ell$-tuples $(a_1, \ldots, a_\ell) \in A^{\ell}$ with $a_1 + \cdots + a_\ell = z$. If $Z \subseteq \Gamma$, then we define the \emph{$\ell$-weight} of $Z$ as
    \[w_{\ell}(Z) := \sum_{z \in Z} w_\ell(z).\]
\end{definition}
The main difference between this proof and the one given in the previous section is that the $(k,\ell)$-tulips we wish to find in the Cayley digraph must not only avoid vertex collisions, but also avoid certain additive collisions. In particular, the condition~\ref{prop:distinct generic inseparable} corresponds to these tulips being rainbow, and the condition~\ref{prop:different subset sums generic inseparable} is more delicate and cannot be entirely encoded in the Cayley digraph itself, even with colors. This means that we will require a stronger notion of robustness compared to the one given in \Cref{def:robust graph}.

\begin{definition}\label{def:robust inseparable}
    Let $\Gamma$ be an Abelian group, let $A$ be a fixed finite subset of $\Gamma$, let $\ell$ be a positive integer, and let $\delta > 0$. We say that an element $z \in \Gamma$ is \emph{$(\delta, \ell)$-robust} with respect to $A$ if the following hold:
    \stepcounter{propcounter}
\begin{enumerate}[label = {{{\normalfont\textbf{\Alph{propcounter}\arabic{enumi}}}}}]
    \item\label{prop:robust partial sum inseparable} For every $y \in \Gamma$ and proper nonempty subset $S \subseteq [\ell]$, less than a $\delta$-proportion of the representations of $z$ as $z = a_1 + \cdots + a_\ell$ with $a_1, \ldots, a_\ell \in A$ have $\sum_{i \in S} a_i = y$; that is,
    \[w_{|S|}(y)w_{\ell-|S|}(z-y) < \delta w_\ell(z).\]
    \item\label{prop:robust dissociated inseparable} More than a $(1- \delta)$-proportion of the representations of $z$ as $z = a_1 + \cdots + a_\ell$ with $a_1, \ldots, a_\ell \in A$ have that $a_1, \ldots, a_\ell$ are distinct and $\{a_1, \ldots, a_\ell\}$ is dissociated.
    \end{enumerate}
\end{definition}

We will now prove our main lemma for this section, analogous to \Cref{lem:robust bipartite graph for digraph}. Here the counting argument differs slightly from that of \Cref{lem:robust bipartite graph for digraph} in order to obtain better bounds on the constants involved.

\begin{lemma}\label{lem:robust bipartite graph inseparable}
    For every $\delta > 0$ and integers $b \geq 1$ and $k \geq 2$, there exist positive constants $\eps = \eps(k, b, \delta)$ and $\beta = \beta(k, b, \delta)$ such that the following holds for every Abelian group $\Gamma$ and every finite subset $A \subseteq \Gamma$ satisfying $|A^{+k}| \leq \eps|A|^k$. There exist an integer $\ell$ with $2 \leq \ell \leq k$ and two subsets $L, R \subseteq \Gamma$ satisfying:
    \stepcounter{propcounter}
    \begin{enumerate}[label = {{{\normalfont\textbf{\Alph{propcounter}\arabic{enumi}}}}}]
        \item\label{prop:right robust inseparable} Every element of $R$ is $(\delta, \ell)$-robust with respect to $A$.
        \item\label{prop:weight count inseparable} $w_\ell(R) \geq \beta|A|^{\ell}$.
        \item\label{prop:right degree inseparable} For every $z \in R$, there are at least $b$ elements $y \in L$ with $z-y \in A$.
        \item\label{prop:left degree inseparable} For every $y \in L$, there are at least $\beta |A|$ elements $z \in R$ with $z-y \in A$.
    \end{enumerate}
\end{lemma}
\begin{proof}
    We may assume $\delta \leq 1$. Define parameters $K_1 < K_2 < \cdots < K_k$ by 
    \[K_\ell := \frac{\delta \cdot 16^\ell(k+1)^\ell(\delta^{-1} + b)^{\ell}}{16(k + \ell)}\]
    for all $1 \leq \ell \leq k$. The exact expression for $K_\ell$ will not be used beyond the facts that $K_1 > 1$, that 
    \begin{equation}\label{eq:exponential growth inseparable}
        K_\ell \geq 16bK_{\ell-1} \qquad \text{for all $\ell \geq 2$},
    \end{equation}
    and that 
    \begin{align}
        \sum_{t=1}^{\ell-1} K_tK_{\ell-t} &= \sum_{t=1}^{\ell-1} \frac{\delta^2 \cdot 16^\ell (k+1)^{\ell}(\delta^{-1} + b)^\ell}{16^2(k+t)(k+\ell-t)} \nonumber \\
        &\leq \frac{\delta^2 \cdot 16^\ell (k+1)^{\ell}(\delta^{-1} + b)^\ell}{16^2}\frac{\ell-1}{(k+1)(k+\ell-1)} \nonumber \\
        &= \frac{\delta}{16}K_\ell\frac{(\ell-1)(k+\ell)}{(k+1)(k+\ell-1)} \nonumber \\
        &\leq \frac{\delta}{16}K_\ell
         \qquad \text{for all $\ell \geq 2$}\label{eq:key recurrence inseparable},
    \end{align}
    Define $\eps := 1/K_k$. Let $\Gamma$ be an Abelian group, and let $A \subseteq \Gamma$ be a finite subset satisfying $|A^{+k}| \leq \eps|A|^k$. By choosing a smaller constant $\eps$ if necessary, we may assume that $|A|$ is sufficiently large. Let $\ell \in [k]$ be minimum with the following property:
    \[\label{eq:star inseparable}\text{There exists $Z \subseteq \Gamma$ of size $|Z| \leq \eps |A|^{\ell}$ with $\ell$-weight $w_{\ell}(Z) \geq K_\ell \eps |A|^\ell$.}\tag{$*$}\]
Such an $\ell$ exists since $|A^{+k}| \leq \eps |A|^k$ by assumption and $w_k(A^{+k}) = |A|^k = K_k \eps |A|^k$. On the other hand, we must have $\ell \geq 2$ because $K_1 > 1$, and $w_1(Z) \leq |Z|$ for any $Z \subseteq \Gamma$. Let $Z \subseteq \Gamma$ be the subset guaranteed by~\eqref{eq:star inseparable} for this choice of $\ell$.

    Let $\hat R$ be the set of elements $z \in Z$ that are $(\delta, \ell)$-robust with respect to $A$. Our first goal is to show that at least half of the assumed $\ell$-weight of $Z$ comes from $\hat R$.
    \begin{claim}\label{cla:robust weight inseparable}
        $w_{\ell}(\hat R) \geq (K_\ell \eps/2)|A|^{\ell}$.
    \end{claim}
    \begin{subproof}
        Let $B_0$ be the set of elements $z \in Z$ such that there are at least $\delta w_{\ell}(z)$ representations of $z$ as $z = a_1 + \cdots + a_\ell$  either with $a_1, \ldots, a_\ell \in A$ non-distinct, or with $\{a_1, \ldots, a_\ell\}$ not dissociated. Since the total number of ordered $\ell$-tuples $(a_1, \ldots, a_\ell) \in A^\ell$ with $\sum_{i \in S} a_i = \sum_{j \in T} a_j$ for some disjoint pair of subsets $S,T \subseteq [\ell]$ with $S \cup T \neq \emptyset$ is at most $3^\ell|A|^{\ell-1}$, we have that
        \begin{equation}\label{eq:degenerate weight inseparable}
            w_\ell(B_0) \leq \delta^{-1} 3^{\ell}|A|^{\ell-1} \leq \frac{K_\ell \eps}{4}|A|^\ell.
        \end{equation}
        Here we used that $|A|$ is sufficiently large.
        
        Note that for every element $z \in Z \setminus (\hat R \cup B_0)$, by \Cref{def:robust inseparable}, there exist an integer $t = t_z \in [\ell-1]$ and two elements $y = y_z$, $v = v_z$ such that $y + v = z$ and $w_t(y)w_{\ell-t}(v) \geq \delta w_\ell(z)$. If more than one such triple $(t,y,v)$ exists, then we may designate one such triple as $(t_z, y_z, v_z)$ arbitrarily. Now, for each $t \in [\ell-1]$, define
        \[B_t := \{z \in Z \setminus (\hat R \cup B_0): t_z = t\},\]
        and note that $ B_1 \cup \cdots \cup B_{\ell-1}$ forms a partition of $Z \setminus (\hat R \cup B_0)$. In order to prove the claim, we will bound $w_\ell(B_t)$ for each $t \in [\ell-1]$.

    To begin, we define, for each $1 \leq t \leq \ell-1$, the set 
    \begin{equation}\label{eq:large weight elements inseparable}
        Y_t := \{y \in \Gamma : w_t(y) \geq 2K_t\},
    \end{equation}
        
        noting that $Y_t$ has size $|Y_t| < \eps|A|^t$ and $t$-weight 
        \begin{equation}\label{eq:weight of large elements inseparable}
            w_t(Y_t) < K_t \eps|A|^t
        \end{equation}
        by our choice of $\ell$ as the minimum satisfying~\eqref{eq:star inseparable}. Now fix $1 \leq t \leq \ell-1$, and consider the $\ell$-weight of $B_t$. By the definition of $B_t$, we have that
        \begin{equation}\label{eq:weight of Bt initial inseparable}
            w_{\ell}(B_t) = \sum_{z \in B_t} w_{\ell}(z) \leq \delta^{-1}\sum_{z \in B_t} w_t(y_z)w_{\ell-t}(v_z).
        \end{equation}
        Crucially, a pair $(y,v) \in \Gamma^2$ can be realized at most once as $(y_z,v_z)$ for some $z \in B_t$ since $y_z + v_z = z$ must be satisfied. Therefore, we can bound the rightmost sum of~\eqref{eq:weight of Bt initial inseparable} by
        \begin{equation}\label{eq:weight of Bt modified}
            \sum_{z \in B_t} w_t(y_z)w_{\ell-t}(v_z) \leq \sum_{\substack{y,v \in \Gamma : \\ y + v \in B_t}} w_t(y)w_{\ell-t}(v).
        \end{equation}
        By \eqref{eq:weight of large elements inseparable}, the definition of $Y_t$ and $Y_{\ell-t}$, and the fact that $|B_t| \leq |Z| \leq \eps |A|^{\ell}$ by~\eqref{eq:star inseparable}, this is at most
        \begin{align*}
        &\sum_{\substack{y \in Y_t, v \in \Gamma : \\ y + v \in B_t}} w_t(y)w_{\ell-t}(v) + \sum_{\substack{y \in \Gamma, v \in Y_{\ell-t} : \\ y + v \in B_t}} w_t(y)w_{\ell-t}(v) + \sum_{\substack{y \in \Gamma \setminus Y_t,v \in \Gamma \setminus Y_{\ell-t} : \\ y + v \in B_t}} w_t(y)w_{\ell-t}(v) \\
        &\leq w_t(Y_t)w_{\ell-t}(\Gamma) + w_t(\Gamma)w_{\ell-t}(Y_{\ell-t}) + |B_t|\left(2K_t\right)\left(2K_{\ell-t}\right) \\
        &\leq \left(K_t + K_{\ell-t} + 4K_tK_{\ell-t}\right)\eps|A|^{\ell} \\
        &\leq 16K_tK_{\ell-t}\eps|A|^{\ell}.
        \end{align*}
        Now combining this with \eqref{eq:weight of Bt initial inseparable} and \eqref{eq:weight of Bt modified} gives
        \[w_\ell(B_t) \leq 16 \delta^{-1} K_tK_{\ell-t}\eps|A|^{\ell}.\]
        
        Summing over all $1 \leq t \leq \ell-1$, we have
        \begin{align*}
            w_\ell(Z \setminus (\hat R \cup B_0)) &= \sum_{t=1}^{\ell-1}w_\ell(B_t) \\
            &\leq \sum_{t=1}^{\ell-1} 16\delta^{-1}K_tK_{\ell-t}\eps|A|^{\ell} \\
            &\leq \frac{K_\ell \eps}{4}|A|^{\ell}
        \end{align*}
        by~\eqref{eq:key recurrence inseparable}.
        Thus, by~\eqref{eq:star inseparable} and~\eqref{eq:degenerate weight inseparable},
    \[w_\ell(\hat R) = w_\ell(Z) - w_\ell(B_0) - w_\ell(Z \setminus (\hat R \cup B_0)) \geq \frac{K_\ell \eps}{2}|A|^{\ell},\]
    as claimed.
    \end{subproof}

    Now let $\hat L := A^{+(\ell-1)}$ be the set of elements $y \in \Gamma$ with positive $(\ell-1)$-weight, and note that $|\hat L| \leq |A|^{\ell-1}$. Also recall that $|\hat R| \leq |Z| \leq\eps|A|^{\ell}$ by~\eqref{eq:star inseparable}. We construct an auxiliary bipartite graph $\hat H$ on $(\hat L, \hat R)$ by joining $y \in \hat L$ to $z \in \hat R$ whenever $z-y \in A$. To show that $\hat H$ has many edges, first recall the definition of the set $Y_{\ell-1} = \{y : w_{\ell-1}(y) \geq 2K_{\ell-1}\}$ from~\eqref{eq:large weight elements inseparable} in the proof of \Cref{cla:robust weight inseparable}. We have
    \[w_\ell(\hat R) = \sum_{y \in \hat L} w_{\ell-1}(y)\deg_{\hat H}(y) \leq \sum_{y \in Y_{\ell-1}}w_{\ell-1}(y)\deg_{\hat H}(y) + \sum_{y \in \hat L \setminus Y_{\ell-1}}2K_{\ell-1}\deg_{\hat H}(y).\]
    Since $\hat H$ has maximum degree at most $|A|$, we can use~\eqref{eq:weight of large elements inseparable} to bound the first sum by
    \[\sum_{y \in Y_{\ell-1}}w_{\ell-1}(y)\deg_{\hat H}(y) \leq w_{\ell-1}(Y_{\ell-1})|A| \leq K_{\ell-1}\eps|A|^{\ell},\]
    whereas the second sum is at most $2K_{\ell-1}e(\hat H)$. Thus, by \Cref{cla:robust weight inseparable} and~\eqref{eq:exponential growth inseparable}, we have
    \begin{equation}\label{eq:edges in first graph inseparable}
        e(\hat H) \geq \frac{w_\ell(\hat R)}{2K_{\ell-1}} - \frac{\eps}{2}|A|^\ell \geq \left(\frac{K_\ell}{4K_{\ell-1}} - \frac{1}{2}\right)\eps|A|^\ell \geq 3b \eps|A|^\ell.
    \end{equation}

    Next, we iteratively delete from $\hat H$ any vertex in $\hat L$ with degree less than $b \eps|A|$ and any vertex in $\hat R$ with degree less than $b$, as long as possible. Let $H$ be the resulting graph, and observe that the number of edges deleted is less than 
    \[b \eps|A||\hat L| + b|\hat R| \leq  2b \eps|A|^{\ell},\]
    so $e(H) \geq b \eps|A|^{\ell}$ by~\eqref{eq:edges in first graph inseparable}. In particular, $H$ is nonempty, and every vertex in $L := \hat L \cap V(H)$ has degree at least~$b \eps|A|$, and every vertex in $R := \hat R \cap V(H)$ has degree at least $b$. Moreover, because every vertex in $L$ has positive $(\ell-1)$-weight, it follows that
    \[w_{\ell}(R) \geq \sum_{y \in L}w_{\ell-1}(y)\deg_H(y) \geq e(H) \geq b \eps |A|^{\ell}.\]
    Thus~\ref{prop:right robust inseparable}--\ref{prop:left degree inseparable} are satisfied with~$\beta := b \eps$.
\end{proof}
We are now ready to prove \Cref{thm:inseparable small growth}.
\begin{proof}[Proof of \Cref{thm:inseparable small growth}]
Choose $\delta := 18^{-k-1}$ and $b := 3^{2k}$. Let $\eps = \eps(k,b,\delta)$ and $\beta = \beta(k,b,\delta)$ be the constants given by \Cref{lem:robust bipartite graph inseparable}. Let $\Gamma$ be an Abelian group, and let $A$ be a subset of $\Gamma$ with $|A^{+k}| \leq \eps|A|^k$. By choosing a smaller constant $\eps$ if necessary, we may assume that $|A|$ is sufficiently large. By \Cref{lem:robust bipartite graph inseparable}, there exist an integer $2 \leq \ell \leq k$ and two subsets $L,R \subseteq \Gamma$ satisfying~\ref{prop:right robust inseparable}--\ref{prop:left degree inseparable}.

We attempt to construct many inseparable solutions to~\eqref{eq:equal k-sums} in $A$ as follows. Initialize $U := \emptyset$.
\begin{enumerate}
    \item Choose distinct elements $a_1, \ldots, a_\ell \in A$ so that $\{a_1, \ldots, a_{\ell}\}$ is dissociated and so that $z_0 := a_1 + \cdots + a_\ell \in R$. Update $U \leftarrow \{a_1, \ldots, a_\ell\}$.
    \item For $1 \leq i \leq k-\ell$, in turn, choose an element $a_{k+\ell+i} \in A \setminus \spn^{\pm}(U)$ so that $y_i := z_{i-1} - a_{k+\ell+i} \in L$, then choose an element $a_{\ell+i} \in A \setminus \spn^{\pm}(U \cup \{a_{k+\ell+i}\})$ so that $z_i := y_i + a_{\ell+i} \in R$. Update $U \leftarrow U \cup \{a_{\ell+i}, a_{k+\ell+i}\}$.
    \item Choose distinct elements $a_{k+1}, \ldots, a_{k+\ell} \in A$ so that $\{a_{k+1}, \ldots, a_{k+\ell}\}$ is dissociated, so that $a_{k+1} + \cdots + a_{k+\ell} = z_{k-\ell}$, and so that
    \[\sum_{j \in T} a_{k+j} \notin \spn^{\pm}(U)\]
    for every proper nonempty subset $T \subseteq [\ell]$.
\end{enumerate}

First, we will verify that $(x_1, \ldots, x_{2k}) = (a_1, \ldots, a_{2k})$ produced in this way is an inseparable solution to~\eqref{eq:equal k-sums}. By design, we have that
\[\sum_{i=1}^\ell a_i + \sum_{i=1}^{k-\ell} (a_{k+\ell+i}-a_{\ell+i}) = z_{k-\ell} = \sum_{i=1}^\ell a_{k+i},\]
which can be rearranged as
\[a_1 + \cdots + a_k = a_{k+1} + \cdots + a_{2k},\]
so $(x_1, \ldots, x_{2k}) = (a_1, \ldots, a_{2k})$ satisfies~\eqref{eq:equal k-sums}. We must now show that this solution is inseparable.

For convenience, define $U_0 = \{a_i : i \in [k] \cup [k+\ell+1,2k]\}$ to be the value of $U$ at the beginning of Step~$3$. By Steps $1$ and $2$, we have that $a_1, \ldots, a_k, a_{k+\ell+1}, \ldots, a_{2k}$ are distinct (and in fact, $U_0$ is dissociated), and Step $3$ ensures that $a_{k+1}, \ldots, a_{k+\ell}$ are distinct from each other and from all elements of $U_0$. Thus~\ref{prop:distinct generic inseparable} is satisfied.

To verify~\ref{prop:different subset sums generic inseparable}, suppose we have two subsets $S,T \subseteq [k]$ with $\sum_{i \in S} a_i = \sum_{j \in T} a_{k+j}$. Let $T_0 := T \cap [\ell+1,k]$, and let $T_1 := T \setminus T_0$. Then
\[\sum_{j \in T_1} a_{k+j} = \sum_{i \in S} a_i - \sum_{j \in T_0} a_{k+j} \in \spn^{\pm}(U_0),\]
so Step 3 ensures that $T_1 = \emptyset$ or $T_1 = [\ell]$. If $T_1 = \emptyset$, then because $U_0$ is dissociated by Steps~$1$ and~$2$, it follows that $S = T_0 = \emptyset$ as well, and thus $S = T = \emptyset$. If $T_1 = [\ell]$, then
\[\sum_{i \in S} a_i - \sum_{j \in T_0} a_{k+j} = \sum_{j=1}^{\ell} a_{k+j} = \sum_{i=1}^k a_i - \sum_{j=\ell+1}^{k}a_{k+j}.\]
Canceling like terms from both sides gives
\[\sum_{i\in[k] \setminus S} a_i = \sum_{j\in [\ell+1,k] \setminus T_0}a_{k+j}.\]
Again, since $U_0$ is dissociated, it follows that $S = [k]$ and $T_0 = [\ell+1,k]$, and hence $S = T = [k]$. Thus~\ref{prop:different subset sums generic inseparable} holds.

    Now let us ensure that this process can be carried out and count the number of choices we have at each of these steps.
    
    \textbf{Step 1.} By~\ref{prop:right robust inseparable}, every element in $R$ is $(\delta,\ell)$-robust, so by~\ref{prop:robust dissociated inseparable}, there are at least $(1-\delta)w_{\ell}(R)$ ordered $\ell$-tuples $(a_1, \ldots, a_\ell) \in A^\ell$ with $a_1, \ldots, a_\ell$ distinct, with $\{a_1, \ldots, a_\ell\}$ dissociated, and with $a_1 + \cdots + a_\ell \in R$. By~\ref{prop:weight count inseparable}, this accounts for at least 
    \begin{equation}\label{eq:step 1 count inseparable}
        (1-\delta)\beta |A|^\ell \geq \frac{\beta}{2} |A|^{\ell}
    \end{equation}
    choices for $a_1, \ldots, a_\ell$ in Step~$1$. 

    \textbf{Step 2.} Fix $1 \leq i \leq k-\ell$. Suppose we have already chosen $a_{\ell+1}, \ldots, a_{\ell+i-1}, a_{k+\ell+1}, \ldots, a_{k+\ell+i-1}$ (and in particular have defined $z_{i-1} \in R$), and we wish to choose $a_{k+\ell+i}$ and $a_{\ell+i}$. Since $z_{i-1} \in R$, we have by~\ref{prop:right degree inseparable} that there are at least $b = 3^{2k}$ choices of $y_i \in L$ such that $a_{k+\ell+i} := z_{i-1} - y_i \in A$. Of these, at most $3^{|U|} < b$ are in $\spn^{\pm}(U)$, leaving us at least one good choice for $a_{k+\ell+i}$. Similarly, by~\ref{prop:left degree inseparable}, there are at least $\beta |A| - 3^{2k}$ good choices for $a_{\ell+i}$. Altogether, this gives at least
    \begin{equation}\label{eq:step 2 count inseparable}
        \left(\beta |A| - 3^{2k}\right)^{k-\ell} \geq \left(\frac{\beta}{2}\right)^{k-\ell}|A|^{k-\ell}
    \end{equation}
    choices for $a_{\ell+1}, \ldots, a_k, a_{k+\ell+1}, \ldots, a_{2k}$ in Step 2.

    \textbf{Step 3.} Since $z_{k-\ell} \in R$, we have by~\ref{prop:right robust inseparable} that $z_{k-\ell}$ is $(\delta, \ell)$-robust. Therefore, by~\ref{prop:robust partial sum inseparable}, less than $2^{\ell}3^{|U|}\delta w_\ell(z_{k-\ell}) \leq w_\ell(z_{k-\ell})/2$ representations of $z_{k-\ell}$ as $z_{k-\ell} = a_{k+1} + \cdots + a_{k+\ell}$ have $\sum_{i \in S} a_{k+i} \in \spn^{\pm}(U)$ for some proper nonempty subset $S \subseteq [\ell]$. Also, by~\ref{prop:robust dissociated inseparable}, less than $\delta w_\ell(z_{k-\ell}) \leq w_\ell(z_{k-\ell})/2$ such representations have $a_{k+1}, \ldots, a_{k+\ell}$ not distinct or have $\{a_{k+1}, \ldots, a_{k+\ell}\}$ not dissociated. This leaves a positive number of choices for $a_{k+1}, \ldots, a_{k+\ell}$ in Step 3.

    Now putting it altogether, we have by~\eqref{eq:step 1 count inseparable} and~\eqref{eq:step 2 count inseparable} that there are at least 
    \[\frac{\beta}{2} |A|^{\ell} \cdot \left(\frac{\beta}{2}\right)^{k-\ell}|A|^{k-\ell} = \left(\frac{\beta}{2}\right)^{k-\ell+1}|A|^k\]
    ways to choose an inseparable solution $(x_1, \ldots, x_{2k}) = (a_1, \ldots, a_{2k})$ to~\eqref{eq:equal k-sums} using Steps 1--3.
\end{proof}

\section{Additively generic solutions}\label{sec:additively generic}

In this section, we finish the proof of \Cref{thm:additively generic integers} and \Cref{thm:inseparable and additively generic finite group} by proving \Cref{thm:truly additively generic small growth}, which finds many additively generic solutions to~\eqref{eq:equal k-sums} in finite subsets of $2$-torsion-free Abelian groups with small $k$-wise sumsets. We will follow the same proof template that we established in \Cref{sec:graphs} for graphs and employed in \Cref{sec:inseparable} for inseparable solutions to~\eqref{eq:equal k-sums}. Though the overall approach is the same, the details are a fair bit more technical now that we require solutions to~\eqref{eq:equal k-sums} that are much more generic. In particular, our proof will require more variations of the weight function which account for \emph{signed} sums of elements in our finite subset $A \subseteq \Gamma$.

\begin{definition}\label{def:weight additive}
Given an abelian group $\Gamma$, a finite subset $A \subseteq \Gamma$, a positive integer $\ell$, and a sequence $\bar{\sigma} = (\sigma_1, \ldots, \sigma_\ell) \in \{\pm 1\}^\ell$ of signs, the \emph{$\bar{\sigma}$-weight} of an element $z \in \Gamma$, denoted $w_{\bar{\sigma}}(z)$, is the number of ordered $\ell$-tuples $(a_1, \ldots, a_\ell) \in A^\ell$ with $\sum_{i=1}^\ell \sigma_i a_i = z$. Similarly, if $Z \subseteq \Gamma$ is a set of elements, then we define the \emph{$\bar{\sigma}$-weight} of $Z$ as
\[w_{\bar{\sigma}}(Z) := \sum_{z \in Z} w_{\bar{\sigma}}(z).\]
\end{definition}

The robustness condition that we require must also be strengthened in order to avoid all types of additive collisions in our solutions to~\eqref{eq:equal k-sums}.

\begin{definition}\label{def:robust additive}
    Let $\Gamma$ be an Abelian group, let $A$ be a fixed finite subset of $\Gamma$, let $\ell$ be a positive integer, let $\bar \sigma = (\sigma_1, \ldots, \sigma_\ell) \in \{\pm 1\}^\ell$, and let $\delta > 0$. We say that an element $z \in \Gamma$ is \emph{$(\delta,\bar \sigma)$-robust} with respect to $A$ if, for every $y \in \Gamma$ and every nonzero $\bar \tau = (\tau_1, \ldots, \tau_\ell) \in \{-1, 0, 1\}^\ell$ which is not $\pm \bar \sigma$, less than a $\delta$-proportion of the representations of $z$ as $z = \sigma_1 a_1 + \cdots + \sigma_{\ell}a_\ell$ with $a_1, \ldots, a_\ell \in A$ have $\sum_{i=1}^\ell \tau_i a_i = y$; that is, the number of solutions $(x_1, \ldots, x_\ell) \in A^{\ell}$ to the system
    \[\left\{\begin{array}{l l}
    \sum_{i=1}^\ell \sigma_i x_i = z \\
    \sum_{i=1}^\ell \tau_i x_i = y
    \end{array}\right.\]
    is less than $\delta w_{\bar \sigma}(z)$.
\end{definition}
We now prove our main lemma, analogous to \Cref{lem:robust bipartite graph inseparable}. The main challenge here compared to previous sections is to make the counting argument work that finds a small set of robust elements with large weight. One convenient aspect of the previous robustness definitions was that they could be expressed in terms of a multiplicative condition on weights of lower complexity (see~\ref{prop:robust vertex} and~\ref{prop:robust partial sum inseparable}). Since \Cref{def:robust additive} cannot be expressed this way, we will need a much more careful argument.

\begin{lemma}\label{lem:robust bipartite graph additive}
    For every $\delta > 0$ and integers $b \geq 1$ and $k \geq 2$, there exist positive constants $\eps = \eps(k, b, \delta)$ and $\beta = \beta(k, b, \delta)$ such that the following holds for every Abelian group $\Gamma$ with no $2$-torsion and every finite subset $A \subseteq \Gamma$ satisfying $|A^{+k}| \leq \eps|A|^k$. There exist an integer $\ell$ with $2 \leq \ell \leq k$, a sequence $\bar \sigma = (\sigma_1, \ldots, \sigma_{\ell}) \in \{\pm 1\}^\ell$, and two subsets $L, R \subseteq \Gamma$ satisfying:
    \stepcounter{propcounter}
    \begin{enumerate}[label = {{{\normalfont\textbf{\Alph{propcounter}\arabic{enumi}}}}}]
        \item\label{prop:right robust additive} Every element of $R$ is $(\delta, \bar \sigma)$-robust with respect to $A$.
        \item\label{prop:weight count additive} $w_{\bar \sigma}(R) \geq \beta|A|^{\ell}$.
        \item\label{prop:right degree additive} For every $z \in R$, there are at least $b$ elements $y \in L$ with $z-y \in \sigma_{\ell}A$.
        \item\label{prop:left degree additive} For every $y \in L$, there are at least $\beta |A|$ elements $z \in R$ with $z-y \in \sigma_{\ell}A$.
    \end{enumerate}
\end{lemma}
\begin{proof}
    We may assume $\delta \leq 1$. For $1 \leq \ell \leq k$, define constants $1 \geq \gamma_k > \cdots > \gamma_1 > \eps_1 > \cdots > \eps_k > 0$ by $\gamma_\ell := (\delta/(4 \cdot 3^k))^{k-\ell}$ and $\eps_\ell := (\gamma_1/(6b))^{\ell}$.  Define $\eps := \eps_k$. Let $\Gamma$ be an Abelian group with no $2$-torsion, and let $A \subseteq \Gamma$ be a finite subset of $\Gamma$ satisfying $|A^{+k}| \leq \eps |A|^k$. By choosing a smaller constant $\eps$ if necessary, we may assume that $|A|$ is sufficiently large. Let $\ell \in [k]$ be minimum with the following property:
    \[\label{eq:star additive}\text{There exists $Z \subseteq \Gamma$ of size $|Z| \leq \eps_\ell|A|^{\ell} $ with $\bar \sigma$-weight $w_{\bar \sigma}(Z) \geq \gamma_\ell |A|^\ell$ for some $\bar \sigma \in \{\pm 1\}^\ell$}.\tag{$*$}\]
    Such an $\ell$ exists since $|A^{+k}| \leq \eps |A|^k$ by assumption and $w_{\bar \sigma}(A^{+k}) = |A|^k = \gamma_k |A|^k$ for $\bar \sigma = (1, \ldots, 1)$.
    On the other hand, we must have $\ell \geq 2$ because $\eps_1 < \gamma_1$, and the $(\sigma)$-weight of any set $Z \subseteq \Gamma$ is at most $|Z|$ for any $\sigma \in \{\pm 1\}$. Let $Z \subseteq \Gamma$ and $\bar \sigma = (\sigma_1, \ldots, \sigma_\ell) \in \{\pm 1\}^\ell$ be the set and sequence guaranteed by~\eqref{eq:star additive} for this choice of $\ell$.

    Let $\hat R$ be the set of elements of $Z$ which are $(\delta, \bar \sigma)$-robust with respect to $A$. Our first goal is to show that at least half of the assumed $\bar \sigma$-weight of $Z$ comes from $\hat R$.
    \begin{claim}\label{cla:robust weight additive}
        $w_{\bar \sigma}(\hat R) \geq (\gamma_\ell/2)|A|^{\ell}$.
    \end{claim}
    \begin{subproof}

    Note that by \Cref{def:robust additive}, for every element $z \in Z \setminus \hat R$, there exists a nonzero sequence of coefficients $\bar \tau = \bar \tau_z = (\tau_1, \ldots, \tau_\ell) \in \{-1, 0, 1\}^\ell$ which is not $\pm \bar \sigma$, and an element $y = y_z \in \Gamma$ such that there are at least $\delta w_{\bar \sigma}(z)$ solutions $(x_1, \ldots, x_\ell) \in A^\ell$ to 
        \[\left\{\begin{array}{l l}
    \sum_{i=1}^\ell \sigma_i x_i = z \\
    \sum_{i=1}^\ell \tau_i x_i = y.
    \end{array}\right.\]
    If more than one such pair $(\bar \tau, y)$ exists, then we may designate one such pair as $(\bar \tau_z, y_z)$ in such a way that $\bar \tau_z$ is supported on a proper nonempty subset of $[\ell]$. This is because $\Gamma$ has no $2$-torsion, so whenever $\bar \tau \in \{\pm 1\}^\ell$, the two derived equations $\sum_{i=1}^\ell (\sigma_i + \tau_i)x_i = z + y$ and $\sum_{i=1}^\ell (\sigma_i - \tau_i)x_i = z - y$ can be reduced to $\sum_{i=1}^\ell \frac{\sigma_i + \tau_i}{2}x_i = y'$ and $\sum_{i=1}^\ell \frac{\sigma_i - \tau_i}{2}x_i = y''$, respectively for some $y',y'' \in \Gamma$. Moreover, we have that $\frac{\sigma_i + \tau_i}{2}, \frac{\sigma_i - \tau_i}{2} \in \{-1, 0, 1\}$ for all $i \in [\ell]$, and there exists some $i \in [\ell]$ such that $\frac{\sigma_i + \tau_i}{2} = 0$ or $\frac{\sigma_i - \tau_i}{2} = 0$.

     Now, for each $\bar \tau \in \{-1,0,1\}^\ell$ supported on a proper nonempty subset of $[\ell]$, define
        \[B_{\bar \tau} := \{z \in Z \setminus \hat R : \bar \tau_z = \bar \tau\},\]
        and note that the sets $B_{\bar \tau}$ form a partition of $Z \setminus \hat R$. In order to prove the claim, we will bound the $\bar \sigma$-weight of each part.

        To that end, fix such a sequence $\bar \tau = (\tau_1, \ldots, \tau_\ell)$. For simplicity, we may assume that the support of~$\bar \tau$ is~$[t]$ for some $1 \leq t \leq \ell-1$. We will abuse notation and write $\bar \tau$ for $(\tau_1, \ldots, \tau_t) \in \{\pm 1\}^t$. Define a set $P \subseteq B_{\bar \tau} \times \Gamma$ consisting of all pairs of the form $(z,y_z)$ for some $z \in B_{\bar \tau}$. Further define, for each $P' \subseteq P$, the quantity $w_{\bar \sigma, \bar \tau}(P')$ to be the number of sequences $(a_1, \ldots , a_\ell) \in A^\ell$ with $(\sum_{i=1}^\ell \sigma_i a_i, \sum_{i=1}^t \tau_i a_i) \in P'$. Note that
    \begin{equation}\label{eq:weighted point set additive}
        w_{\bar \sigma, \bar \tau}(P) = \sum_{z \in B_{\bar \tau}} \left|\left\{(a_1, \ldots, a_\ell) \in A^\ell : \sum_{i=1}^\ell \sigma_i a_i = z, \sum_{i=1}^t \tau_i a_i = y_z\right\}\right|\geq \delta w_{\bar \sigma}(B_{\bar \tau})
    \end{equation}
    by how we chose $y_z$ for each $z \in B_{\bar \tau}$. Therefore, it will be enough for us to bound $w_{\bar \sigma, \bar \tau}(P)$, which we do ``fiber by fiber'', as follows. For each $y \in \Gamma$, consider the fiber
    \[P_y := \{(z,y) \in P : y_z = y\},\]
    and observe that these fibers partition $P$. We now differentiate the fibers by their size. Define
    \begin{align*}
        Y_{\mathrm{small}} &:= \{y \in \Gamma : |P_y| \leq \eps_{\ell-t}|A|^{\ell-t}\}, \\
        Y_{\mathrm{large}} &:= \{y \in \Gamma : |P_y| > \eps_{\ell-t}|A|^{\ell-t}\}.
    \end{align*}
    It will also be convenient before we move forward to define the \emph{potential} of a sequence $(a_1, \ldots, a_t) \in A^t$, denoted by $\mathrm{pot}(a_1, \ldots, a_t)$, to be the number of sequences $(a_{t+1}, \ldots, a_\ell)$ such that $\left(\sum_{i=1}^\ell \sigma_i a_i, \sum_{i=1}^t \tau_i a_i\right) \in P$. Note that the potential of any sequence is trivially at most $|A|^{\ell-t}$, and that
    \begin{equation*}
        w_{\bar \sigma, \bar \tau}(P) = \sum_{y \in \Gamma} w_{\bar \sigma, \bar \tau}(P_y) =  \sum_{y \in \Gamma} \sum_{\substack{(a_1, \ldots, a_t) \in A^t : \\ \sum_{i=1}^t\tau_i a_i = y}}\mathrm{pot}(a_1, \ldots, a_t).
    \end{equation*}

    We handle large fibers first.
    Note that $|P| = |B_{\bar \tau}| \leq |Z| \leq \eps_\ell|A|^{\ell}$ by~\eqref{eq:star additive}, so 
    \[|Y_{\mathrm{large}}| \leq \frac{|P|}{\eps_{\ell-t}|A|^{\ell-t}} \leq \frac{\eps_\ell}{\eps_{\ell-t}}|A|^t = \eps_t|A|^t.\]
    It then follows by our choice of $\ell$ as the minimum satisfying~\eqref{eq:star additive} that $w_{\bar \tau}(Y_{\mathrm{large}}) \leq \gamma_t|A|^t$. Therefore, the contribution of these large fibers to $w_{\bar \sigma, \bar \tau}(P)$ is at most
    \begin{align}
        \sum_{y \in Y_{\mathrm{large}}} w_{\bar \sigma, \bar \tau}(P_y) &= \sum_{y \in Y_{\mathrm{large}}}\sum_{\substack{(a_1, \ldots, a_t) \in A^t : \\ \sum_{i=1}^t\tau_i a_i = y}}\mathrm{pot}(a_1, \ldots, a_t) \nonumber \\
        &\leq \sum_{y \in Y_{\mathrm{large}}}w_{\bar \tau}(y)|A|^{\ell-t} \nonumber \\
        &= w_{\bar\tau}(Y_{\mathrm{large}})|A|^{\ell-t} \nonumber \\
        &\leq \gamma_t|A|^\ell. \label{eq:large fibers additive}
    \end{align}

    Next, we consider the small fibers. Fix $y \in Y_{\mathrm{small}}$ and $\bar a = (a_1, \ldots, a_t) \in A^t$ with $\sum_{i=1}^t \tau_i a_i = y$. Define 
    \begin{align*}
        V_{\bar a} &:= \left\{z - \sum_{i=1}^t\sigma_ia_i : (z,y) \in P_y\right\}, \\
        \bar \sigma' &:= (\sigma_{t+1}, \ldots, \sigma_\ell),
    \end{align*}
    and note that we can write the potential of $\bar a$ as
    \begin{align*}\mathrm{pot}(a_1, \ldots, a_t) &= \left|\left\{(a_{t+1}, \ldots, a_\ell) \in A^{\ell-t} : \left(\sum_{i=1}^\ell \sigma_i a_i, y\right) \in P_y\right\}\right| \\
    &= \left|\left\{(a_{t+1}, \ldots, a_\ell) \in A^{\ell-t} : \sum_{i=t+1}^\ell \sigma_i a_i \in V_{\bar a}\right\}\right| \\
    &= w_{\bar \sigma'}(V_{\bar a}).
    \end{align*}
    Now since $|V_{\bar a}| = |P_y| \leq \eps_{\ell-t}|A|^{\ell-t}$ since $y \in Y_{\mathrm{small}}$, we again have by the minimality of $\ell$ satisfying~\eqref{eq:star additive} that 
    \[\mathrm{pot}(\bar a) = w_{\bar \sigma'}(V_{\bar a}) \leq \gamma_{\ell-t}|A|^{\ell-t}.\]
    Now summing over all $(a_1,\ldots, a_t) \in A^t$ with $\sum_{i=1}^t \tau_i a_i \in Y_{\mathrm{small}}$ gives
    \begin{align*}
        \sum_{y \in Y_{\mathrm{small}}} w_{\bar \sigma, \bar \tau}(P_y) &= \sum_{\substack{(a_1, \ldots, a_t) \in A^t : \\ \sum_{i=1}^t\tau_i a_i \in Y_{\mathrm{small}}}}\mathrm{pot}(a_1, \ldots, a_t) \\
        &\leq \sum_{\substack{(a_1, \ldots, a_t) \in A^t : \\ \sum_{i=1}^t\tau_i a_i \in Y_{\mathrm{small}}}} \gamma_{\ell-t}|A|^{\ell-t} \\
        &\leq \gamma_{\ell-t}|A|^\ell.
    \end{align*}

    Combining this with~\eqref{eq:weighted point set additive} and~\eqref{eq:large fibers additive} gives
    \begin{align*}
        w_{\bar \sigma}(B_{\bar \tau}) &\leq \delta^{-1}w_{\bar \sigma, \bar \tau}(P) \\
        &= \delta^{-1} \left(\sum_{y \in Y_{\mathrm{large}}}w_{\bar \sigma, \bar \tau}(P_y) + \sum_{y \in Y_{\mathrm{small}}}w_{\bar \sigma, \bar \tau}(P_y) \right) \\
        &\leq \delta^{-1}(\gamma_t + \gamma_{\ell-t})|A|^\ell \\
        &\leq 2\delta^{-1}\gamma_{\ell-1}|A|^\ell
    \end{align*}
    since $\gamma_1 \leq \cdots \leq \gamma_{\ell-1}$.

    Finally, we can sum over all possible choices of $\bar \tau \in \{-1,0,1\}^\ell$ supported on a nonempty proper subset of~$[\ell]$ to obtain
    \[w_{\bar \sigma}(\hat R) = w_{\bar \sigma}(Z) - \sum_{\bar \tau}w_{\bar \sigma}(B_{\bar \tau}) \geq \gamma_\ell |A|^{\ell} - 3^{\ell}(2\delta^{-1}\gamma_{\ell-1})|A|^{\ell} \geq \frac{\gamma_{\ell}}{2}|A|^{\ell}\]
    by~\eqref{eq:star additive}, as claimed. 
    \end{subproof}
    
    Now let $\hat L$ be the set of elements $y \in \Gamma$ with positive $\bar \sigma'$-weight for $\bar \sigma' := (\sigma_1, \ldots, \sigma_{\ell-1})$ (that is, $\hat L$ is the $(\ell-1)$-wise sumset $\sigma_1A + \cdots + \sigma_{\ell-1}A$), and note that $|\hat L| \leq |A|^{\ell-1}$. Also recall that $|\hat R| \leq |Z| \leq \eps_\ell|A|^{\ell}$ by~\eqref{eq:star additive}. We construct an auxiliary bipartite graph $\hat H$ on $(\hat L, \hat R)$ by joining $y \in \hat L$ to $z \in \hat R$ whenever $z - y \in \sigma_{\ell}A$. 
    
    We want to give a lower bound on the number of edges of $\hat H$. To that end, let $Y := \{y \in \hat L : w_{\bar \sigma'}(y) \geq 2\gamma_{\ell-1}/\eps_{\ell-1}\}$, and note that $Y$ cannot be larger than $\eps_{\ell-1}|A|^{\ell-1}$ by our choice of $\ell$ as the smallest satisfying~\eqref{eq:star additive}. Consequently, $w_{\bar \sigma'}(Y) \leq \gamma_{\ell-1}|A|^{\ell-1}$, also by our choice of $\ell$. Therefore, since $\hat H$ has maximum degree at most $|A|$, we have
    \begin{align*}
        w_{\bar \sigma}(\hat R) &= \sum_{y \in \hat L} w_{\bar \sigma'}(y)\deg_{\hat H}(y) \\
        &\leq w_{\bar \sigma'}(Y)\Delta(\hat H) + \sum_{y \in \hat L \setminus Y} \frac{2\gamma_{\ell-1}}{\eps_{\ell-1}}\deg_{\hat H}(y) \\
        &\leq \gamma_{\ell-1}|A|^{\ell} + \frac{2\gamma_{\ell-1}}{\eps_{\ell-1}}e(\hat H).
    \end{align*}
    
    Rearranging and applying \Cref{cla:robust weight additive}, we have
    \begin{equation}\label{eq:edges in first graph additive}
        e(\hat H) \geq \frac{\eps_{\ell-1}w_{\bar \sigma}(\hat R)}{2\gamma_{\ell-1}} - \frac{\eps_{\ell-1}}{2}|A|^\ell \geq \left(\frac{\eps_{\ell-1}\gamma_{\ell}}{4\gamma_{\ell-1}} - \frac{\eps_{\ell-1}}{2}\right)|A|^\ell \geq 3b\eps_{\ell}|A|^\ell.
    \end{equation}
    Here we used that $\gamma_{\ell-1} \leq \gamma_{\ell}/4$ and $\eps_{\ell} \leq \eps_{\ell-1}/(6b)$.

    Next, we iteratively delete from $\hat H$ any vertex in $\hat L$ with degree less than $b\eps_{\ell}|A|$ and any vertex in $\hat R$ with degree less than $b$, as long as possible. Let $H$ be the resulting graph, and observe that the number of edges deleted is less than 
    \[b \eps_\ell|A||\hat L| + b|\hat R| \leq  2b\eps_{\ell}|A|^{\ell},\]
    so $e(H) \geq b \eps_{\ell}|A|^{\ell}$ by~\eqref{eq:edges in first graph additive}. In particular, $H$ is nonempty, every vertex in $L := \hat L \cap V(H)$ has degree at least~$b\eps_{\ell}|A|$, and every vertex in $R := \hat R \cap V(H)$ has degree at least $b$. Moreover, because every vertex in $L$ has positive $\bar \sigma'$-weight, it follows that
    \[w_{\bar \sigma}(R) \geq \sum_{y \in L}w_{\bar \sigma'}(y)\deg_H(y) \geq e(H) \geq b\eps_{\ell} |A|^{\ell}.\]
    Lastly, every element of $R \subseteq \hat R$ is $(\delta, \bar \sigma)$-robust by the definition of $\hat R$.
    Thus~\ref{prop:right robust additive}--\ref{prop:left degree additive} are satisfied with~${\beta := b\eps_k}$.
\end{proof}

We are now ready to prove \Cref{thm:truly additively generic small growth}. 
\begin{proof}[Proof of \Cref{thm:truly additively generic small growth}]
Choose $\delta := 3^{-3k}$ and $b := 5^{2k}$. Let $\eps = \eps(k,b,\delta)$ and $\beta = \beta(k,b,\delta)$ be the constants given by \Cref{lem:robust bipartite graph additive}. Let $\Gamma$ be an Abelian group with no $2$-torsion, and let $A$ be a subset of $\Gamma$ with $|A^{+k}| \leq \eps|A|^k$. By choosing a smaller constant $\eps$ if necessary, we may assume that $|A|$ is sufficiently large. By \Cref{lem:robust bipartite graph additive}, there exist an integer $2 \leq \ell \leq k$, a sequence $\bar \sigma = (\sigma_1, \ldots, \sigma_{\ell}) \in \{\pm 1\}^\ell$, and two subsets $L,R \subseteq \Gamma$ satisfying~\ref{prop:right robust additive}--\ref{prop:left degree additive}.

We attempt to construct many additively generic solutions to~\eqref{eq:equal k-sums} in $A$ as follows. Initialize $U := \emptyset$.
\begin{enumerate}
    \item Choose distinct elements $a_1, \ldots, a_\ell \in A$ so that $\{a_1, \ldots, a_{\ell}\}$ is strongly dissociated and so that $z_0 := \sigma_1 a_1 + \cdots + \sigma_{\ell}a_{\ell} \in R$. Update $U \leftarrow \{a_1, \ldots, a_\ell\}$.
    \item For $1 \leq i \leq k-\ell$, in turn, choose an element $a_{k+\ell+i} \in A$ so that $a_{k+\ell+i}, 2a_{k+\ell+i} \notin \spn^{\pm 2}(U)$ and $y_i := z_{i-1} - \sigma_{\ell}a_{k+\ell+i} \in L$, then choose an element $a_{\ell+i} \in A$ so that $a_{\ell+i}, 2a_{\ell+i} \notin \spn^{\pm 2}(U \cup \{a_{k+\ell+i}\})$ and $z_i := y_i + \sigma_{\ell}a_{\ell+i} \in R$. Update $U \leftarrow U \cup \{a_{\ell+i}, a_{k+\ell+i}\}$.
    \item Choose elements $a_{k+1}, \ldots, a_{k+\ell} \in A$ so that $\sigma_1a_{k+1} + \cdots + \sigma_{\ell}a_{k+\ell} = z_{k-\ell}$, and so that
    \[\sum_{i=1}^\ell \tau_i a_{k+i} \notin \spn^{\pm}(U)\]
    for every nonzero sequence $\bar \tau = (\tau_1,\ldots, \tau_\ell) \in \{-1,0,1\}^\ell$ which is not $\pm \bar \sigma$.
\end{enumerate}

First, we will verify that $(x_1, \ldots, x_{2k}) = (a_1, \ldots, a_{2k})$ produced in this way is an additively generic solution to the equation
\begin{equation}\label{eq:mixed additive}
    \sum_{i=1}^\ell \sigma_i x_i + \sum_{i=\ell+1}^k \sigma_\ell x_i = \sum_{i=1}^\ell \sigma_i x_{k+i} + \sum_{i=\ell+1}^k \sigma_\ell x_{k+i},
\end{equation}
and therefore some fixed permutation of $(a_1, \ldots, a_{2k})$ (obtained by transposing $a_i$ and $a_{k+i}$ whenever $\sigma_{\min(i,\ell)} = -1$) is an additively generic solution to~\eqref{eq:equal k-sums}.

To begin, note that by construction,
\begin{equation}\label{eq:common sum additive}
    \sum_{i=1}^\ell \sigma_i a_i + \sum_{i=1}^{k-\ell}(\sigma_\ell a_{\ell+i} - \sigma_\ell a_{k+\ell+i}) = z_{k-\ell} = \sum_{i=1}^\ell \sigma_i a_{k+i},
\end{equation}
so $(x_1, \ldots, x_{2k}) = (a_1, \ldots, a_{2k})$ satisfies~\eqref{eq:mixed additive}.

Next, to show that this solution is additively generic, we make some further definitions and observations.
Define $S_0 := [k] \cup [k+\ell+1,2k]$, so that $U_0 := \{a_i : i \in S_0\}$ is the value of $U$ at the beginning of Step~$3$. Note that because $\Gamma$ has no $2$-torsion, Steps~$1$ and~$2$ guarantee that $U_0$ is a strongly dissociated set consisting of $2k-\ell$ distinct elements. Also, for simplicity, define $\bar \sigma' = (\sigma_1', \ldots, \sigma_{2k}')$ by
    \[\sigma_i' = \left\{\begin{array}{r l}
     \sigma_i & \text{if $1 \leq i \leq \ell$}, \\
     \sigma_\ell & \text{if $\ell + 1 \leq i \leq k$}, \\
     -\sigma_i & \text{if $k+1 \leq i \leq k + \ell$}, \\
     -\sigma_\ell & \text{if $k + \ell + 1 \leq i \leq 2k$,}
     \end{array}\right.\]
     so that~\eqref{eq:mixed additive} can be written as $\sum_{i=1}^{2k} \sigma_i' a_i = 0$.
     Now suppose that we have 
     $\sum_{i=1}^{2k}\tau_i' a_i = 0$ for some $\bar \tau' = (\tau_1', \ldots, \tau_{2k}') \in \{-1, 0, 1\}^{2k}$. We wish to show that $\bar \tau' = 0$ or $\bar \tau' = \pm \bar \sigma'$. Let $\bar \tau = (\tau_1, \ldots, \tau_\ell) := (\tau_{k+1}', \ldots, \tau_{k+\ell}')$.
     Then
     \begin{equation}\label{eq:tau additive}\sum_{i=1}^\ell \tau_i a_{k+i} = \sum_{i=k+1}^{k+\ell}\tau_i' a_i = \sum_{i \in S_0} -\tau_i' a_i \in \spn^{\pm}(U_0),
     \end{equation}
     so by Step 3, $\bar \tau$ must be either $0$ or $\bar \sigma$ or $-\bar \sigma$. If $\bar \tau = 0$, then $\sum_{i \in S_0} \tau_i' a_i = 0$ by~\eqref{eq:tau additive},
     so because $U_0$ is strongly dissociated (hence dissociated), we must have $\tau_i' = 0$ for all $i \in S_0$, and thus $\bar \tau' = 0$. If $\bar \tau = -\bar \sigma$, then by~\eqref{eq:common sum additive} and~\eqref{eq:tau additive}, we have
     \[\sum_{i \in S_0} \tau_i' a_i = \sum_{i=1}^\ell \sigma_i a_{k+i} = \sum_{i=1}^\ell \sigma_i a_i + \sum_{i=1}^{k-\ell} (\sigma_\ell a_{\ell+i} - \sigma_\ell a_{k+\ell+i}) = \sum_{i \in S_0}\sigma_i' a_i.\]
     Now because $\tau_i' - \sigma_i' \in \{-2, -1, 0, 1, 2\}$ for all $i \in S_0$, and $U_0$ is strongly dissociated, we have that $\tau_i' = \sigma_i'$ for all $i \in S_0$, and thus $\bar \tau' = \bar \sigma'$. A similar argument shows that if $\bar \tau = \bar \sigma$, then $\bar \tau' = -\bar \sigma'$. This proves that $(x_1, \ldots, x_{2k}) = (a_1, \ldots, a_{2k})$ is an additively generic solution to~\eqref{eq:mixed additive} for any $a_1, \ldots, a_{2k}$ obtained by Steps~$1$--$3$.

     Now let us ensure that this process can be carried out and count the number of choices we have at each of these steps.

     \textbf{Step 1.} By~\ref{prop:weight count additive}, there are at least $w_{\bar \sigma}(R) \geq \beta |A|^{\ell}$ choices of $a_1, \ldots, a_{\ell} \in A$ with $\sigma_1 a_1 + \cdots + \sigma_\ell a_\ell \in R$. Of these, we claim that at most $5^{\ell}|A|^{\ell-1}$ have $a_1, \ldots, a_\ell$ not distinct or $\{a_1, \ldots, a_\ell\}$ not strongly dissociated. Indeed, if one of these two situations occurs, then there exists a nonzero sequence $\bar \tau = (\tau_1, \ldots, \tau_\ell) \in \{-2,-1,0,1,2\}^{\ell}$ such that $\sum_{i=1}^\ell \tau_i a_i = 0$, and so because $\Gamma$ has no $2$-torsion, the element $a_i$ with $i := \max(\mathrm{supp}(\bar \tau))$ is determined by the remaining $\ell-1$ elements $a_1, \ldots, a_{i-1}, a_{i+1}, \ldots, a_\ell$. Therefore, the number of choices of $a_1, \ldots, a_\ell \in A$ with $\sum_{i=1}^\ell \tau_i a_i = 0$ is at most $|A|^{\ell-1}$. Summing over all nonzero sequences $\bar \tau \in \{-2,-1,0,1,2\}^\ell$ gives at most $5^\ell|A|^{\ell-1}$ bad choices for $a_1, \ldots, a_\ell$, as claimed. Altogether, since we assumed $|A|$ was large, this gives at least
     \begin{equation}\label{eq:step 1 count additive}
         \beta |A|^{\ell} - 5^{\ell}|A|^{\ell-1} \geq \frac{\beta}{2}|A|^{\ell}
     \end{equation}
     valid choices for $a_1, \ldots, a_\ell$ in Step~1.

     \textbf{Step 2.} Fix $1 \leq i \leq k-\ell$. Suppose we have already chosen $a_{\ell+1}, \ldots, a_{\ell+i-1}, a_{k+\ell+1}, \ldots, a_{k+\ell+i-1}$ (and in particular have defined $z_{i-1} \in R$), and we wish to choose $a_{k+\ell+i}$ and $a_{\ell+i}$. Since $z_{i-1} \in R$, we have by~\ref{prop:right degree additive} that there are at least $b = 5^{2k}$ choices of $y_i \in L$ such that $\sigma_\ell a_{k+\ell+i} := z_{i-1} - y_i \in \sigma_\ell A$. Of these, at most $2 \cdot 5^{|U|} < b$ have $a_{k+\ell+i}$ or $2a_{k+\ell+i}$ in $\spn^{\pm 2}(U)$, leaving us at least one good choice for~$a_{k+\ell+i}$. Similarly, by~\ref{prop:left degree additive}, there are at least $\beta |A| - 5^{2k}$ good choices for $a_{\ell+i}$. Altogether, this gives at least
    \begin{equation}\label{eq:step 2 count additive}
        \left(\beta |A| - 5^{2k}\right)^{k-\ell} \geq \left(\frac{\beta}{2}\right)^{k-\ell}|A|^{k-\ell}
    \end{equation}
    valid choices for $a_{\ell+1}, \ldots, a_k, a_{k+\ell+1}, \ldots,  a_{2k}$ in Step 2.

    \textbf{Step 3.} Since $z_{k-\ell} \in R$, we have by~\ref{prop:right robust additive} that $z_{k-\ell}$ is $(\delta, \bar \sigma)$-robust. Therefore, by \Cref{def:robust additive}, less than $3^{\ell}3^{|U|}\delta w_{\bar \sigma}(z_{k-\ell}) \leq w_{\bar \sigma}(z_{k-\ell})$ representations of $z_{k-\ell}$ as $z_{k-\ell} = \sigma_1a_{k+1} + \cdots + \sigma_\ell a_{k+\ell}$ have $\sum_{i =1}^{\ell} \tau_i a_{k+i} \in \spn^{\pm}(U)$ for some nonzero sequence $\bar \tau = (\tau_1, \ldots, \tau_\ell) \in \{-1,0,1\}^{\ell}$ which is not $\pm \bar \sigma$. This leaves a positive number of valid choices for $a_{k+1}, \ldots, a_{k+\ell}$ in Step 3.

    Now putting it altogether, we have by~\eqref{eq:step 1 count additive} and~\eqref{eq:step 2 count additive} that there are at least 
    \[\frac{\beta}{2} |A|^{\ell} \cdot \left(\frac{\beta}{2}\right)^{k-\ell}|A|^{k-\ell} = \left(\frac{\beta}{2}\right)^{k-\ell+1}|A|^k\]
    ways to choose an additively generic solution to~\eqref{eq:mixed additive}, hence to~\eqref{eq:equal k-sums}, in $A$ using Steps~$1$--$3$.
\end{proof}

\section{Linearly generic solutions to symmetric equations}\label{sec:linearly generic}
Now we move back into the vector space setting to discuss \Cref{thm:linearly generic}, which finds $\mathbb F_q$-linearly generic solutions to some symmetric linear equation~\eqref{eq:symmetric equation} in $2k$ variables in any linear configuration $A \subseteq \mathbb F_q^n$ of size at least $CN^{1/k}$, where $N := q^n$. Our proof of this theorem is nearly identical to the proof we just gave for \Cref{thm:truly additively generic small growth}, which finds additively generic solutions to the equation~\eqref{eq:equal k-sums}. In fact, the only changes that need to be made to that proof are that we must make use of all coefficients in $\mathbb F_q$ instead of merely $\{-1, 0, 1\}$ (and sometimes $\{-2,-1,0,1,2\}$), and we must adjust the parameters accordingly. We omit the details of the proof for the sake of brevity, but we include the key definitions that we use to give the reader a sense of how those details might look.
\begin{definition}
Given a linear configuration $A \subseteq \mathbb F_q^n$, a positive integer $\ell$, and a sequence $\bar{\lambda} = (\lambda_1, \ldots, \lambda_\ell) \in (\mathbb F_q^*)^\ell$ of nonzero coefficients in $\mathbb F_q$, the \emph{$\bar{\lambda}$-weight} of a vector $z \in \mathbb F_q^n$ with respect to $A$, denoted $w_{\bar{\lambda}}(z)$, is the number of ordered $\ell$-tuples $(a_1, \ldots, a_\ell) \in A^\ell$ with $\sum_{i=1}^\ell \lambda_i a_i = z$.
\end{definition}

\begin{definition}\label{def:robust Fq}
    For $\delta > 0$ and $\bar \lambda = (\lambda_1, \ldots, \lambda_\ell) \in (\mathbb F_q^*)^\ell$, we say that $z \in \mathbb F_q^n$ is \emph{$(\delta,\bar \lambda)$-robust} with respect to a linear configuration $A \subseteq \mathbb F_q^n$ if, for every $y \in \mathbb F_q^n$ and every nonzero $\bar \mu = (\mu_1, \ldots, \mu_\ell) \in \mathbb F_q^\ell$ which is not an $\mathbb F_q$-scalar multiple of $\bar \lambda$, we have that the number of solutions $(x_1, \ldots, x_\ell) \in A^\ell$ to the system
    \[\left\{\begin{array}{l l}
    \sum_{i=1}^\ell \lambda_i x_i = z \\
    \sum_{i=1}^\ell \mu_i x_i = y
    \end{array}\right.\]
    is less than $\delta w_{\bar \lambda}(z)$.
\end{definition}
As in the previous proofs, the strategy is to find a small set of vectors of large $\bar \lambda$-weight which are $(\delta,\bar \lambda)$-robust for some small $\delta > 0$ and some sequence $\bar \lambda = (\lambda_1, \ldots, \lambda_\ell)$ of coefficients, then use this to greedily find a solution in a structure resembling a $(k,\ell)$-tulip. We end the discussion here.

\section{Concluding remarks}\label{sec:concluding remarks}

Already, we have discussed many open problems throughout this article, but we would like to mention a few more in closing.

\subsection{Linearly generic solutions to a single equation}

One possible way to strengthen \Cref{thm:linearly generic} in the vector space setting is to guarantee a linearly generic solution to one particular symmetric linear equation~\eqref{eq:symmetric equation}, whose coefficients have been specified in advance. As discussed in the introduction, for certain choices of coefficients, there exist linear configurations much larger than $\Theta(N^{1/k})$ avoiding linearly generic solutions. However, we believe the symmetric additive equation~\eqref{eq:equal k-sums} to be a plausible candidate for a strong bound matching the one given in \Cref{thm:linearly generic}.

\begin{question}
    For every integer $k \geq 2$ and prime power $q$, does there exist a constant $C$ such that for every linear configuration $A \subseteq \mathbb F_q^n$ of size at least $CN^{1/k}$, there exists an $\mathbb F_q$-linearly generic solution $(x_1, \ldots, x_{2k}) \in A^{2k}$ to the equation
    \[x_1 + \cdots + x_k = x_{k+1} + \cdots + x_{2k},\tag{\ref{eq:equal k-sums}}\]
    where $N := q^n$?
\end{question}

We mentioned in the introduction that this is true for every $q$ when $k=2$, which follows from the Sidorenko property for~\eqref{eq:equal k-sums}. We have also seen that it holds for $q \in \{2,3\}$ and every $k \geq 2$ by \Cref{cor:linearly generic 2 and 3}, and it is also tight for $q=2$ by \Cref{thm:bch}. Furthermore, by \Cref{thm:Bk construction vector space}, this would be tight if true whenever the characteristic of $\mathbb F_q$ is greater than $k$. 

In this paper, we have only considered symmetric linear equations, and in particular, we have not discussed linear equations in an odd number of variables (with nonzero coefficients). We expect this type of equation to behave somewhat differently from equations in an even number of variables, though it is not clear to what extent over $\mathbb F_q$ for $q \geq 3$. Over $\mathbb F_2$, any equation in an odd number of variables is not affine-type, and it follows from an $\mathbb F_q$-analogue of the Erd\H{o}s-Stone Theorem due to Geelen and Nelson~\cite{GN15} that the maximum size of a linear configuration $A \subseteq \mathbb F_2^n$ excluding $\mathbb F_2$-linearly generic solutions to any particular equation in an odd number of variables is $(\frac 12 + o(1)) \cdot 2^n$. Over $\mathbb F_q$ for $q \geq 3$, affine-type equations in $2k+1$ variables exist for every $k \geq 1$, and Sauermann~\cite{S23} has shown that such equations are always temperate, meaning that there exists $\delta = \delta_q(k) > 0$ such that $\ex_q(n,F) = O(N^{1-\delta})$, where $F$ is the linear configuration of size $2k+1$ corresponding to the equation, and $N := q^n$. This comes from applying the slice rank method in an iterative fashion. Unfortunately, due to the iterative nature of the argument, the function $\delta_q(k)$ shrinks to~$0$ at a superexponential rate in $k$ for every fixed $q$. We suspect that there is much room for improvement on this bound.

\begin{question}
    For every prime power $q \geq 3$, does there exist a function $\delta = \delta_q(k)$ tending to $1$ as $k \to \infty$ such that every linear configuration $A \subseteq \mathbb F_q^n$ with no $\mathbb F_q$-linearly generic solution to any affine-type $\mathbb F_q$-linear equation
    \[\lambda_1x_1 + \cdots + \lambda_{2k+1}x_{2k+1}=0\]
    with $2k+1$ variables and nonzero coefficients has size $O(N^{1-\delta})$, where $N := q^n$?
\end{question}

\subsection{Lower bounds}
One way to witness the tightness of \Cref{thm:linearly generic} would be to produce large \emph{$2k$-general sets}, which are closely related to linear error-correcting codes over $\mathbb F_q$ of large dimension and minimum distance greater than $2k$. For $k \geq 1$, a \emph{$k$-general set} in $\mathbb F_q^n$ (or equivalently in the projective space $\mathrm{PG}(n-1,q)$) is a set of vectors with the property that any $k$ of them are linearly independent. It is not hard to show that the maximum size of a $2k$-general set in $\mathbb F_q^n$ is at most $O(N^{1/k})$ for every prime power $q$, where $N := q^n$, and this also follows from \Cref{thm:linearly generic}. \Cref{thm:bch} gives $2k$-general sets of asymptotically optimal size $\Omega(N^{1/k})$ in $\mathbb F_2^n$ for every $k \geq 2$. Also, $4$-general sets of size $\Omega(N^{1/2})$ are known to exist in $\mathbb F_q^n$ for $q \in \{3,4\}$ by~\cite{huang2019sidon} and~\cite{pavese2025on}, respectively.
\begin{question}
    For every integer $k \geq 2$ and prime power $q$, do there exist $2k$-general sets in $\mathbb F_q^n$ of size~$\Omega(N^{1/k})$, where $N:= q^n$?
\end{question}

We also know very little about the tightness of \Cref{thm:kst vector space} for most values of the parameters, even in the case $r=2$.

\begin{question}
    For which values of $s$ and $t$ is the K\H{o}v\'ari--S\'os--Tur\'an-type bound $\ex_q(n,\Vect_q[K_{s,t}]) = O(N^{1-1/s})$ tight?
\end{question}
We know the bound is tight for $s=2$ over every finite field. This follows from the existence of large Sidon sets in $\mathbb F_q^n$ when $q$ is odd~\cite{cilleruelo2012combinatorial} and the $k=2$ case of \Cref{thm:bch} when $q$ is even. The simplest open case is when $q=2$, $s=3$, and $t$ is sufficiently large, where we want to know if there exist $\mathbb F_2$-linear configurations $A \subseteq \mathbb F_2^n$ of size $\Omega(2^{2n/3})$ with no linear copy of $\Vect_2[K_{3,t}]$.

\subsection{Matroids excluding an even circuit and a uniform minor}

One of the characterizing features of $\mathbb F_q$-representable matroids is that they exclude lines $U_{2,q+2}$ with $q+2$ elements, not only as restrictions, but also as minors. However, for $q \geq 3$, the class of $U_{2,q+2}$-minor-free matroids is much larger than the class of $\mathbb F_q$-representable matroids, and in general, the former have no underlying algebraic structure. It would be nice to generalize our matroid analogue of the Bondy--Simonovits Theorem (\Cref{cor:matroid bondy-simonovits}) to this larger class of matroids.

\begin{question}
    For every fixed integer $q \geq 2$, is it true that every rank-$n$ simple $U_{2,q+2}$-minor-free matroid with no circuits of size $2k$ has size $O(q^{n/k})$?
\end{question}
 We suspect that this should at least be true when $q$ is a prime power. The methods of this paper rely heavily on linear structure in the matroid, so a completely new approach may be necessary. In~\cite{davies2025girth} (Conjectures 6 and 8), similar problems were posed for forbidden circuits of size at most $g$ when additionally forbidding the cycle matroid of a complete graph and its dual as a minor.

 \subsection{Rainbow \texorpdfstring{Tur\'an}{Turán} problem for even cycles}
 The \emph{rainbow Tur\'an problem} asks for the maximum number of edges in a properly edge-colored host graph excluding a \emph{rainbow} copy of a given graph $H$, meaning that no two edges of $H$ have the same color in the host graph. Formally the \emph{rainbow extremal number} (or \emph{rainbow Tur\'an number}) of a graph $H$, denoted $\ex^*(n,H)$, is the maximum number of edges in a properly edge-colored $n$-vertex graph excluding $H$ as a rainbow subgraph. Keevash, Mubayi, Sudakov, and Verstra\"ete~\cite{keevash2007rainbow} were the first to pose this problem in this level of generality, and they expressed special interest in the case of even cycles, perhaps  motivated by the following relation to $B_k^*$-sets. They observed that \Cref{thm:bose chowla} gives rise to a natural construction of a rainbow $C_{2k}$-free $n$-vertex graph with $(1+o(1))(n/2)^{1+1/k}$ edges. This is to take a $B_k$-set $A$ in $\Gamma := \mathbb Z_{q^{k}-1}$ of size $q$, where $q$ is a prime power satisfying $q^{k}-1 = (1-o(1))n/2$, and consider the \emph{edge-colored Cayley bipartite graph} generated by $A$, whose vertex set consists of two disjoint copies $\Gamma_1$ and $\Gamma_2$ of $\Gamma$, with an edge of color $a$ from $g \in \Gamma_1$ to $g + a \in \Gamma_2$ for every $a \in A$ and $g \in \Gamma_1$. It is not hard to see that such a graph is rainbow $C_{2k}$-free because $A$ is a $B_k^*$-set, so this establishes $\ex^*(n,C_{2k}) = \Omega(n^{1+1/k})$. It is also worth noting that when $n$ is a power of $2^k$, then a similar, but non-bipartite, Cayley graph construction in $\mathbb F_2^m$ using \Cref{thm:bch} yields a slightly larger $n$-vertex, rainbow $C_{2k}$-free graph, with $(\frac 12-o(1))n^{1+1/k}$ edges. In~\cite{keevash2007rainbow}, the authors conjectured that a matching upper bound should hold, namely $\ex^*(n,C_{2k}) = O(n^{1+1/k})$, and they proved this upper bound for $k=3$.  The full conjecture was later confirmed by Janzer~\cite{janzer2023rainbow}, who proved that $\ex^*(n,C_{2k}) \leq 2^{ck^2}n^{1+1/k}$ for some constant $c$. By considering edge-colored Cayley bipartite graphs, these gave the first combinatorial proofs of the upper bound of $O(N^{1/k})$ for the maximum size of a $B_k^*$-set in an arbitrary finite Abelian group of order $N$. The bound on the leading constant was subsequently improved by Kim, Lee, Liu, and Tran~\cite{kim2024rainbow}, who showed that $\ex^*(n,C_{2k}) \leq 10^5 k^3 n^{1+1/k}$. Both of these results were also accompanied by the corresponding supersaturation results for rainbow even cycles.  

We have seen that our methods have recovered both this $O(N^{1/k})$ bound for $B_k^*$-sets (Theorems~\ref{thm:additively generic integers} and~\ref{thm:inseparable and additively generic finite group}) and the $O(n^{1+1/k})$ bound on $\ex(n,C_{2k})$ from the Bondy--Simonovits Theorem. Though it seems reasonable that our proofs could be modified to the rainbow Tur\'an setting as well to reprove Janzer's result $\ex^*(n,C_{2k}) = O(n^{1+1/k})$, we have not found a way to avoid color collisions in our cycles except in the Cayley graph setting. 

\vspace{10pt}
\noindent
\textbf{Statement of AI use.} All the statements and the proofs were developed by the authors in the traditional way, without any use of AI tools. The final
version of the paper was written entirely by the authors. The authors used ChatGPT (GPT-6 Astra)
for assistance with proofreading, identifying
relevant literature, and  also to stylize our original figures.

\vspace{10pt}
\noindent
\textbf{Acknowledgements.} We would like to thank Dingding Dong for pointing us to the paper~\cite{cho2024norming}, which connects Sidorenko notions for (hyper)graphs and vector spaces.

\bibliographystyle{abbrv}
\bibliography{literature}

\end{document}